\documentclass[10pt,psamsfonts]{amsart}
\usepackage{amsmath}
\usepackage{amsthm}
\usepackage{amssymb}
\usepackage{amscd}
\usepackage{amsfonts}
\usepackage{amsbsy}
\usepackage{graphicx}
\usepackage[dvips]{psfrag}
\usepackage{array}
\usepackage{color}
\usepackage{epsfig}
\usepackage{url}
\usepackage{overpic}
\usepackage{epstopdf}
\usepackage[normalem]{ulem}
\usepackage{float}
\usepackage{enumerate}
\usepackage{multirow}
\usepackage{rotating}
\usepackage{hyperref}

\newcommand{\R}{\ensuremath{\mathbb{R}}}

\newcommand{\F}{\ensuremath{\mathcal{F}}}
\newcommand{\CF}{\ensuremath{\mathcal{F}}}
\newcommand{\z}{\ensuremath{\mathcal{Z}}}
\newcommand{\Ge}{\ensuremath{\mathcal{G}}}
\newcommand{\He}{\ensuremath{\mathcal{H}}}
\newcommand{\CO}{\ensuremath{\mathcal{O}}}
\newcommand{\la}{\lambda}

\newcommand{\f}{\varphi}
\newcommand{\al}{\alpha}
\newcommand{\be}{\beta}

\newcommand{\s}{\Sigma}

\newcommand{\sgn}{\mathrm{sign}}

\DeclareMathOperator{\Span}{Span}

\def\p{\partial}
\def\e{\varepsilon}

\newtheorem {theorem} {Theorem}

\newtheorem {proposition} [theorem]{Proposition}

\newtheorem {lemma}  [theorem]{Lemma}

\newtheorem {remark} [theorem]{Remark}

\newtheorem {mtheorem} {Theorem}

\newtheorem {claim}{Claim}

\usepackage{mathpazo}

\begin{document}
\renewcommand{\arraystretch}{1.5}

\title[Melnikov analysis for planar piecewise linear vector fields]{Higher order Melnikov analysis for planar piecewise linear vector fields with nonlinear switching curve}

\author[K. da S. Andrade]{Kamila da S. Andrade$^1$}
\address{$^1$ Instituto de Matem\'{a}tica e Estat\'{i}stica, Universidade Federal de Goi\'as\\
	Campus II, Samambaia, \  74001-970, Goi\^ania, GO, Brazil.}
\email{kamila.andrade@ufg.br}

\author[O.A.R. Cespedes]{Oscar A. R. Cespedes$^2$}
\address{$^2$ Departamento de Matem\'{a}tica, Universidade
	Federal de Vi\c{c}osa,\ Campus Universit\'ario,
	36570-000, Vi\c cosa, MG, Brazil} 
\email{oscar.ramirez@ufv.br} 

\author[D. R. Cruz]{Dayane R. Cruz$^3$}
\author[D.D. Novaes]{Douglas D. Novaes$^3$}
\address{$^3$Departamento de Matem\'{a}tica, Instituto de Matem\'{a}tica, Estat\'{i}stica e Computa\c{c}\~{a}o Cient\'{i}fica, Universidade
Estadual de Campinas, \ Rua S\'{e}rgio Buarque de Holanda, 651, Cidade Universit\'{a}ria Zeferino Vaz, 13083-859, Campinas, SP,
Brazil} 
\email{dayanemat29@gmail.com, ddnovaes@unicamp.br}



\maketitle

\begin{abstract}
In this paper, we are interested in providing lower estimations for the maximum number of limit cycles $H(n)$ that planar piecewise linear differential systems with two zones separated by the curve $y=x^n$ can have, where $n$ is a positive integer. For this, we perform a higher order Melnikov analysis for piecewise linear perturbations of the linear center. In particular, we obtain that $H(2)\geq 4,$ $H(3)\geq 8,$ $H(n)\geq7,$ for $n\geq 4$ even, and  $H(n)\geq9,$ for $n\geq 5$ odd. This improves all the previous results for $n\geq2.$ Our analysis is mainly based on some recent results about Chebyshev systems with positive accuracy and Melnikov theory, which will be developed at any order for a class of nonsmooth differential systems with nonlinear switching manifold.
\end{abstract}

\section{Introduction}

Recently, the interest in nonsmooth differential systems has grown mainly due to the amount of engineering, physical, biological, and real processes problems that are naturally modeled by this class of differential systems (see, for instance,  \cite{BerBuddCham2008} and the references therein for piecewise linear differential models of real processes). Much of the questions on nonsmooth differential systems arise as extensions of classical and important results already established for smooth differential systems. Since these questions appear naturally in many applications, they are not merely mathematical or academic (see, for instance, \cite{ColJefLazOlm2017,MakLam2012,Sim2010,Teixeira2011}). 

Motivated by the second part of the 16th Hilbert's Problem, there exists an increasing interest on establishing a uniform upper bound for the maximum number of limit cycles that planar piecewise linear differential systems can have. In the research literature, one can find many papers addressing this problem assuming that the switching curve is a straight line (see, for instance, \cite{ArtLliMedTei2014,BPT13,BraMel2013,FrePonTor2012,FrePonTor2014,GianPli2001,HanZha2010,HuaYan2013,HuaYan2014,Lh10,Li2014,LliNovTei2015b,LNT15b,LliPon2012,LliTeiTor2013}, and references therein). In this case, no examples with more than $3$ limit cycles are known so far. In \cite{BraMel2014,NovPon2015}, it is shown that such an upper bound is strictly related to the nonlinearity of the switching curve. In this direction, piecewise linear system with two zones separated by a curve $y=x^n,$ with $n$ being a positive integer, has been addressed (see, for instance, \cite{BBLN19,LMN15,LliZha2019}).

Accordingly, given a positive integer $n$, let $H(n)$ denote the maximum number of limit cycles that planar piecewise linear systems with two zones separated by the curve $y=x^n$ can have. In this paper, we are interested in determining lower bounds for $H(n).$ For that, we consider the following planar piecewise linear vector field	
\begin{equation}\label{general-system1}
Z(x,y)=\left\{  \begin{array}{lc}
X(x,y) = \left(\begin{array}{c}
y+\displaystyle\sum_{i=1}^{k}\e^iP_i^+(x,y) \\
-x+\displaystyle\sum_{i=1}^{k}\e^iQ_i^+(x,y)
\end{array}\right), & y-x^n>0, \\ & \\
Y(x,y) = \left(\begin{array}{c}
y+\displaystyle\sum_{i=1}^{k}\e^iP_i^-(x,y)\\
-x+\displaystyle\sum_{i=1}^{k}\e^iQ_i^-(x,y)
\end{array}\right), & y-x^n<0, 
\end{array}
\right.
\end{equation}
where $n$ is a positive integer, and $P_i^{\pm}$ and $Q_i^{\pm}$ are affine functions provided by
\[ \begin{array}{rcl}
P_i^+(x,y) & = &a_{0i}+a_{1i}x+a_{2i}y, \\
P_i^-(x,y) & = &\al_{0i}+\al_{1i}x+\al_{2i}y, \\
Q_i^+(x,y) & = &b_{0i}+b_{1i}x+b_{2i}y, \\
Q_i^-(x,y) & = &\be_{0i}+\be_{1i}x+\be_{2i}y,
\end{array}\]
with $a_{ji},\al_{ji},b_{ji},\be_{ji}\in\R,$ for $i\in\{1,\ldots,k\}$ and $j\in\{0,1,2\}$. The switching curve of system \eqref{general-system1} is provided by $\Sigma=\{(x,y)\in\R^2:\, y=x^n\}$.
Here, we assume the Filippov's convention \cite{filippov2013differential} for trajectories of \eqref{general-system1}.

Usually, periodic solutions of differential systems are studied by means of Poincar\'{e} maps. 
 Since system \eqref{general-system1} is a $k$-order perturbation of the linear center $(x',y')=(y,-x),$ it is easy to see that, for $|\e|$ sufficiently small, a Poincar\'{e} Map $\pi_{\e}$ can be defined in the section $S=\{(x,y):\,x>0,\,y=0\},$ which is parameterized by $x.$ Moreover, for $|\e|$ sufficiently small, $(x;\e)\mapsto\pi_{\e}(x)$ is smooth (because it is composition of smooth functions), thus we can compute the Taylor expansion of $\pi_{\e}$ around $\e=0$ as
\[
\pi_{\e}(x)=x+\sum_{i=1}^{k} \e^i M_i(x)+\CO(\e^{k+1}).
\]
For each $i\in\{1,\ldots,k\},$ the function $M_i$ is called {\it Melnikov function of order $i$}. Denote $M_0=0$ and let $M_{\ell},$ for some $\ell\in\{1,2,\ldots,k\},$ be the first non-vanishing Melnikov function, that is $M_i= 0$ for $i\in\{0,\ldots,\ell-1\}$ and $M_{\ell}\neq0.$ Since periodic solutions of \eqref{general-system1} are in one-to-one correspondence with fixed points of  the Poincar\'{e} map $\pi_{\e}$, one can easily get as a simple consequence of the {\it Implicity Function Theorem} that simple zeros of $M_{\ell}$ correspond to limit cycles of \eqref{general-system1}. Accordingly, we denote by $m_{\ell}(n)$ the maximum number of simple zeros that the first non-vanishing Melnikov function $M_{\ell}$ can have for any choice of parameters $a_{ji},\al_{ji},b_{ji},\be_{ji}\in\R,$ for $i\in\{1,\ldots,\ell\}$ and $j\in\{0,1,2\}$.  

Under certain conditions, upper bounds for the maximum number of limit cycles  of \eqref{general-system1}, bifurcating from the period annulus of \eqref{general-system1}$\big|_{\e=0},$ can also be given based on $m_{\ell}(n)$ (see, for instance, \cite[Theorems 14 and 15]{LRT11} and \cite[Theorems 3.1, 3.2, and 3.3]{HY21}). However, in general, the values $m_{\ell}(n),$ for $\ell\in\{1,\ldots,k\},$ provide lower bounds for $H(n)$, indeed $H(n)\geq m_{\ell}(n)$ for every $\ell\in\{1,\ldots,k\}.$ In \cite{BPT13}, a higher order analysis of system \eqref{general-system1} was performed assuming a straight line as the switching curve, that is $n=1.$ It was shown that $m_1(1)=m_2(1)=1,$ $m_3(1)=2,$ and $m_{\ell}(1)=3$ for $\ell\in\{4,\ldots,7\}.$ The nonlinear case of switching curves was firstly addressed in \cite{LMN15} by means of {\it Averaging Theory}. In particular, it was shown that $m_1(2)=3.$ 
It is worth mentioning that the {\it Averaging Theory} is a classical method to attack this problem (see \cite{CLN17,LNT2014,SVM}), which has been recently developed for nonsmooth differential systems in \cite{ILN2017,LMN15,LliNovRod2017,LNT15} (see, also, \cite{Han17,HS15}). However, in these previous studies some strong conditions are assumed on the switching set when dealing with higher order perturbations. Indeed, in \cite{LMN15} it was observed that the first order averaging function can always be used for determining the number of zeros of the first Melnikov function, however  higher order averaged functions do not always control the bifurcation of isolated periodic solutions for nonsmooth differential systems. Thus, in \cite{BBLN19} the Melnikov functions up to order $2$ was obtained for a wider class of nonsmooth differential systems with nonlinear switching curve. In addition, it was shown that $m_1(3)=3$ and $m_2(3)=7.$ The known values in research literature for $m_{\ell}(n),$ for  $\ell\in\{1,\ldots,6\},$ are summarized in Table \ref{table1}.  In particular, these previous studies provided $H(1)\geq 3$, $H(2)\geq 2,$ and $H(3)\geq 7.$ 
\begin{table}[h!]
\centering 
	 {\bf Known results for $ m_{\ell}(n)$}   \vspace{0.2cm}
	 
\begin{tabular}{cccccc}
 &  & \multicolumn{4}{c}{Order $\ell$} \\ \cline{3-6} 
 & \multicolumn{1}{c|}{} & \multicolumn{1}{c|}{{\bf 1}} & \multicolumn{1}{c|}{{\bf 2}} & \multicolumn{1}{c|}{{\bf 3}} & \multicolumn{1}{c|}{${\bf 4\leq \ell\leq6}$} \\ \cline{2-6} 
\multicolumn{1}{c|}{\multirow{4}{*}{\begin{sideways}Degree $n$\end{sideways}}} & \multicolumn{1}{c|}{{\bf 1}} & \multicolumn{1}{c|}{1} & \multicolumn{1}{c|}{1} & \multicolumn{1}{c|}{2} & \multicolumn{1}{c|}{3} \\ \cline{2-6} 
\multicolumn{1}{c|}{} & \multicolumn{1}{c|}{{\bf 2}} & \multicolumn{1}{c|}{3} & \multicolumn{1}{c|}{--} & \multicolumn{1}{c|}{--} & \multicolumn{1}{c|}{--} \\ \cline{2-6} 
\multicolumn{1}{c|}{} & \multicolumn{1}{c|}{{\bf 3}} & \multicolumn{1}{c|}{3} & \multicolumn{1}{c|}{7} & \multicolumn{1}{c|}{--} & \multicolumn{1}{c|}{--} \\ \cline{2-6} 
\multicolumn{1}{c|}{} & \multicolumn{1}{c|}{${\bf n\geq3}$} & \multicolumn{1}{c|}{--} & \multicolumn{1}{c|}{--} & \multicolumn{1}{c|}{--} & \multicolumn{1}{c|}{--} \\ \cline{2-6} 
\end{tabular}
\bigskip
\caption{Known values in the research literature. In particular, $H(1)\geq 3$, $H(2)\geq 2,$ and $H(3)\geq 7.$ }\label{table1}
\end{table}

Our first main result completes Table \ref{table1} by providing the values $m_{\ell}(n),$ for $\ell\in\{1,\ldots,6\}$ and $n\in\mathbb{N}.$ In particular, we obtain that $H(2)\geq 4,$ $H(3)\geq 8,$ $H(n)\geq7,$ for $n\geq 4$ even, and  $H(n)\geq9,$ for $n\geq 5$ odd, which improves all the previous results for $n\geq2$. The contribution of Theorem \ref{Theorem-Melnikov2} is summarized in Table \ref{table2}. 

     \begin{mtheorem}\label{Theorem-Melnikov2}
  Consider the planar piecewise linear differential system \eqref{general-system1}. For $n\in\mathbb{N}$ and $\ell\in\{1,\ldots,6\},$ we have the following values for $m_{\ell}(n)$:
  \begin{enumerate}[$(i)$]
  \item $m_1(1)=1$, $m_1(2)=3$, $m_1(n)=3$ for $n\geq 3$ odd, and $m_1(n)=4$ for $n\geq4$ even;

  \smallskip
  
  \item $m_2(1)=1$, $m_2(2)=4$, $m_2(n)=7$ for $n\geq3$;
  
  \smallskip
  
  \item $m_3(1)=2$, $m_3(2)=4$, $m_3(n)=7$ for $n\geq3$;
  
  \smallskip
  
  \item  for $\ell\in\{4,5\},$ $m_{\ell}(1)=3$, $m_{\ell}(2)=4$, $m_{\ell}(n)=7$ for $n\geq3$;

 \item $m_6(1)=3$, $m_6(2)=4$, $8\leq m_6(3)\leq 10$,  $m_6(n)=7$ for $n\geq 4$ even, and $9\leq m_6(n)\leq 14$ for $n\geq 5$ odd.
 
  \end{enumerate}
  Consequently, $H(2)\geq 4,$ $H(3)\geq 8,$ $H(n)\geq7,$ for $n\geq 4$ even, and  $H(n)\geq9,$ for $n\geq 5$ odd.
  \end{mtheorem}

  \begin{table}[h]

\centering 
	{\bf Our contribution} \vspace{0.2cm}

\begin{tabular}{cccccccc}
 &  & \multicolumn{6}{c}{Order $k$} \\ \cline{3-8} 
 & \multicolumn{1}{c|}{} & \multicolumn{1}{c|}{{\bf 1}} & \multicolumn{1}{c|}{{\bf 2}} & \multicolumn{1}{c|}{{\bf 3}} &\multicolumn{1}{c|}{{\bf 4}} &\multicolumn{1}{c|}{{\bf 5}} & \multicolumn{1}{c|}{${\bf 6}$} \\ \cline{2-8} 
\multicolumn{1}{c|}{\multirow{4}{*}{\begin{sideways}Degree $n$\end{sideways}}} & \multicolumn{1}{c|}{{\bf 1}} & \multicolumn{1}{c|}{1} & \multicolumn{1}{c|}{1} & \multicolumn{1}{c|}{2} & \multicolumn{1}{c|}{3}& \multicolumn{1}{c|}{3}& \multicolumn{1}{c|}{3} \\ \cline{2-8} 

\multicolumn{1}{c|}{} & \multicolumn{1}{c|}{{\bf 2}} & \multicolumn{1}{c|}{3} & \multicolumn{1}{c|}{4} & \multicolumn{1}{c|}{4} & \multicolumn{1}{c|}{4}& \multicolumn{1}{c|}{4}& \multicolumn{1}{c|}{4} \\ \cline{2-8} 

\multicolumn{1}{c|}{} & \multicolumn{1}{c|}{ ${\bf 3}$} & \multicolumn{1}{c|}{3} & \multicolumn{1}{c|}{7} & \multicolumn{1}{c|}{7} & \multicolumn{1}{c|}{7}& \multicolumn{1}{c|}{7}& \multicolumn{1}{c|}{$8\leq m_6\leq 10$} \\ \cline{2-8} 

\multicolumn{1}{c|}{} & \multicolumn{1}{c|}{${\bf n\geq 4}$ {\bf even}} & \multicolumn{1}{c|}{4} & \multicolumn{1}{c|}{7} & \multicolumn{1}{c|}{7} & \multicolumn{1}{c|}{7}& \multicolumn{1}{c|}{7}& \multicolumn{1}{c|}{7} \\ \cline{2-8} 

\multicolumn{1}{c|}{} & \multicolumn{1}{c|}{ ${\bf n\geq 5}$ {\bf odd}} & \multicolumn{1}{c|}{3} & \multicolumn{1}{c|}{7} & \multicolumn{1}{c|}{7} & \multicolumn{1}{c|}{7}& \multicolumn{1}{c|}{7}& \multicolumn{1}{c|}{$9\leq m_6\leq 14$} \\ \cline{2-8} 
\end{tabular}
\bigskip
	\caption{Our main result competes Table \ref{table1}. In particular, $H(2)\geq 4,$ $H(3)\geq 8,$ $H(n)\geq7,$ for $n\geq 4$ even, and  $H(n)\geq9,$ for $n\geq 5$ odd}\label{table2}
\end{table}  

In order to prove Theorem \ref{Theorem-Melnikov2}, we shall first compute the Melnikov functions up to order $6$ for system \eqref{general-system1}. For that, Theorem \ref{thm:melnikov} provides  the higher order Melnikov functions for a class of nonsmooth differential systems with nonlinear switching manifold, which generalizes the results obtained in \cite{BBLN19} at any order. Some recent results about Extended Chebyshev systems with positive accuracy \cite{NOTOR2017} are also applied to obtain Theorem \ref{Theorem-Melnikov2}.

 This paper is structured as follows. In Section \ref{sec:melnikov} we state our second main result, Theorem \ref{thm:melnikov}, which develop the Melnikov theory at any order for a class of nonsmooth differential systems with nonlinear switching manifold. Theorem \ref{thm:melnikov} is proven in the Appendix. In Section \ref{sec:chebyshev}, we provide some families of Extended Chebyshev systems and Extended Chebyshev systems with accuracy $1$, which are used in Sections \ref{sec:first} and \ref{sec:higher}, together with the Melnikov theory, to prove Theorem \ref{Theorem-Melnikov2}. Statement $(i)$ is proven in Section \ref{sec:first} and  statements $(ii)-(iv)$ are proven in Section \ref{sec:higher}.

\section{Melnikov functions}\label{sec:melnikov}

In this section, we establish the Melnikov functions at any order for a class of nonsmooth differential systems. Consider, an open subset $D\subset\R^d$, $\mathbb{S}^1=\R/T$ for some period $T>0$, and $k$ a positive integer. Let $\theta_j:D\to \mathbb{S}^1$, $j\in\{1,\ldots,N\}$, be $C^{k-1}$ functions such that $\theta_0(x)\equiv0<\theta_1(x)<\cdots<\theta_N(x)<T\equiv\theta_{N+1}(x)$, for all $x\in D$. Under the assumptions above, we consider the following piecewise smooth differential system
\begin{equation}\label{general-system2} 
\dot{x} = \sum_{i=1}^k\e^iF_i(t,x)+\e^{k+1}R(t,x,\e),
\end{equation}
where
\[
F_i(t,x)=\left\{
\begin{array}{ll}
F_i^0(t,x), & 0<t<\theta_1(x), \\
F_i^1(t,x), & \theta_1(x)<t<\theta_2(x), \\
\vdots & \\
F_i^N(t,x), & \theta_N(x)<t<T, 
\end{array}
\right.
\]
and
\[
R(t,x,\e)=\left\{
\begin{array}{ll}
R^0(t,x,\e), & 0<t<\theta_1(x), \\
R^1(t,x,\e), & \theta_1(x)<t<\theta_2(x), \\
\vdots & \\
R^N(t,x,\e), & \theta_N(x)<t<T, 
\end{array}
\right.
\]
with $F_i^j:\mathbb{S}^1\times D\rightarrow\R^d$, $R^j:\mathbb{S}^1\times D\times (-\e_0,\e_0)\rightarrow\R^d$, for $i\in\{1,\ldots,k\}$ and $j\in\{1,\ldots,N\}$, being $\mathcal{C}^{k}$ functions and $T-$periodic in the variable $t$. In this case, the switching manifold is provided by $\s=\{(\theta_i(x),x);\ x\in D,\ i\in\{0,1,\ldots,N\} \}$. For the sake of simplicity, denote
\begin{equation}\label{Fj}
F^j(t,x,\e)=\sum_{i=1}^{k}\e^iF_i^j(t,x)+\e^{k+1}R^i(t,x,\e), \text{ for } j\in\{0,\ldots,N\}.
\end{equation}

It is worth mentioning that the differential system \eqref{general-system2} is a particular form of the differential  systems previously considered in \cite{LMN15}, where first and second order averaging method for detecting periodic solutions of a wider class of nonsmooth systems were developed. This particular class of systems seems to have first appeared in \cite{Han17} and after in \cite{BBLN19}. It is written in a standard form suitable for applying techniques from regular perturbation theory, and then it is very common in the research literature as well as some variations using characteristic function (see also \cite{ILN2017, LliNovRod2017}).

As the main result of this section, Theorem \ref{thm:melnikov} provides sufficient conditions for $T$-periodic solutions $x(t,\varepsilon)$ of system \eqref{general-system2} to be given as simple zeros of the $i$th Melnikov function, 
\begin{equation*}\label{melnikovn}
M_i(x)=\dfrac{1}{i!}z_i^N(T,x),
\end{equation*}
where $z_i^j(t,x)$ is defined recursively for $i=1,\dots,k$ and $j=0,\dots,N$ as follows:
\begin{equation}\label{ztilde}
\begin{array}{rl}
z^0_1(t,x)=&\displaystyle\int_0^{t}F_1^0(s,x)ds,\\
z^j_1(t,x)=&z^{j-1}_1(\theta_j(x),x)+\displaystyle\int_{\theta_j(x)}^{t}F_1^j(s,x)ds,\\
z_i^0(t,x)=&i!\displaystyle\int_{0}^t\left( F_{i}^0(s,x)+\displaystyle\sum_{l=1}^{i-1}\displaystyle\sum_{b\in S_l}\dfrac{1}{b_1!b_2!2!^{b_2}\dots b_l!l!^{b_l}}\partial^{L_b}_xF_{i-l}^0(s,x)\displaystyle\prod_{m=1}^l\left(z_m^0(s,x)\right)^{b_m}\right) ds,\\
z_i^j(t,x)=&z_i^{j-1}(\theta_j(x),x)\\
&+i!\displaystyle\int_{\theta_j(x)}^t\left( F_{i}^j(s,x)+\displaystyle\sum_{l=1}^{i-1}\displaystyle\sum_{b\in S_l}\dfrac{1}{b_1!b_2!2!^{b_2}\dots b_l!l!^{b_l}}\partial^{L_b}_xF_{i-l}^j(s,x)\displaystyle\prod_{m=1}^l\left(z_m^j(s,x)\right)^{b_m}\right) ds\\
&+i!\displaystyle\sum_{p=1}^{i-1}\dfrac{1}{p!}\dfrac{\partial^{p}}{\partial\varepsilon^{p}}\left(\delta_{i-p}^j\left( A_j^p(x,\varepsilon),x\right)  \right) \Big|_{\varepsilon=0},
\end{array}
\end{equation}
where $\delta_i^j(t,x)=\dfrac{1}{i!}\left( z_i^{j-1}(t,x)-z_i^j(t,x)\right) $ and $A^{p}_{j}(x,\varepsilon)=\displaystyle\sum_{q=0}^p\dfrac{\varepsilon^q}{q!}\alpha_j^q(x)$ with
\begin{equation}\label{alphaj}
\alpha_j^q(x)=\displaystyle\sum_{l=1}^q\dfrac{q!}{l!}\displaystyle\sum_{u\in S_{q,l}}D^l\theta_j(x)\left(\displaystyle\prod_{r=1}^l w_{u_r}^j(x)\right),  \text{ for } q=1,\ldots ,k-1,
\end{equation}
and
\begin{equation}\label{omegaj}
\begin{array}{rcl}
w_1^j(x)&=&z_{1}^{j-1}(\theta_j(x),x),\\
w_i^j(x)&=&\dfrac{1}{i!}z_{i}^{j-1}(\theta_j(x),x)\\
& &+\displaystyle\sum_{a=1}^{i-1}\displaystyle\sum_{b\in S_a}\dfrac{1}{(i-a)!b_1!b_2!2!^{b_2}\dots b_a!a!^{b_a}}\partial_t^{L_b}z_{i-a}^{j-1}(\theta_j(x),x)\displaystyle\prod_{m=1}^a\left(\alpha_j^m(x) \right)^{b_m}.
\end{array}
\end{equation}

Here  $\partial_x^{L_b}G(t, x)$ denotes the derivative of order $L_b$ of a function G, with respect to the variable $x$, $S_a$ is the set of all a-tuples of non-negative integers $(b_1, b_2, \dots , b_a)$ satisfying $b_1 + 2b_2 + \dots +
ab_a = a$, $L _b= b_1 + b_2 + \dots + b_a$, and $S_{q,a}$  is the set of all a-tuples of positives integers $(b_1, b_2, \dots , b_a)$ satisfying $b_1 + b_2 + \dots +
b_a = q$. Considering the notations, we are able to enounce our main result on Melnikov functions.

\begin{mtheorem}\label{thm:melnikov}
	Consider the nonsmooth differential system \eqref{general-system2} and denote $M_0=0.$ Assume that, for some $\ell\in\{1,\ldots,k\},$ $M_i=0$, for $i=1,\dots,\ell-1,$ and $M_{\ell}\neq0$. If  $M_{\ell}(a^*)=0$ and $\det(DM_{\ell}(a^*))\neq0,$ for some $a^*\in D,$ then, for $|\e|\neq0$ sufficiently small, there exists a unique $T$-periodic solution $x(t,\e)$ of system \eqref{general-system2} satisfying $x(0,\e)\to a^*$ as $\e\to0$.  
\end{mtheorem}

Theorem \ref{thm:melnikov} generalizes  the results of \cite{BBLN19} and it is proven in the Appendix. Indeed,
applying the recurrence above for $i=1,2$ we get the expressions for $M_1$ and $M_2$ obtained in  \cite{BBLN19}, namely
\begin{equation}\label{d1}
\begin{array}{rcl}
M_1(x)&=&\displaystyle\int_0^TF_1(s,x) ds,\vspace{0.3cm}\\
M_2(x)&=&\displaystyle\int_0^T\left[D_xF_1(s,x)\int_0^sF_1(t,x)dt+F_2(s,x) \right] ds\vspace{0.2cm}\\
& &\displaystyle+\sum_{j=1}^{N}\left( F_1^{j-1}\left(\theta_j(x),x\right)-F_1^{j}\left(\theta_j(x),x\right)\right)  \alpha_j^1(x).
\end{array}
\end{equation}

\section{Chebyshev systems}\label{sec:chebyshev}

Let $\mathcal{F}=\left[u_{0}, \ldots, u_{n}\right]$ be an ordered set of smooth functions defined on the closed interval $[a, b]$ and let \(\operatorname{Span}(\mathcal{F})\) be the set of all linear combinations of elements of $\mathcal{F}$.
The maximum number of zeros, counting multiplicity, that any nontrivial
function in \( \operatorname{Span}(\mathcal{F})\) can have will be denoted by \(\mathcal{Z}(\mathcal{F})\). A classical tool to study $\mathcal{Z}(\CF)$ is the Theory of Chebyshev systems. The set $\CF$ is said to be an {\it Extended Chebyshev} system or just {\it ET-system} on $[a,b]$ if $\mathcal{Z}(\CF)\leq n$ (see \cite{KASTU1966}). 
If the functions in $\mathcal{F}$ are linearly independent, it is always possible to find an element in $\Span(\CF)$ with $n$ zeros (see \cite{LliSwi2011}), in this case $\mathcal{Z}(\CF)=n.$ When $\mathcal{Z}(\CF)= n+k,$ the set $\CF$ is called an ET-system with {\it accuracy }$k$ on $[a,b],$ (see \cite{NOTOR2017}).

Recall that the Wronskian of the ordered set $[u_0, \ldots, u_{s}]$, of $s+1$  functions, is defined as
\begin{equation*}\label{wrons}
W(x)= W(u_0,\ldots, u_s)(x)=\det(M(u_0, \ldots, u_s)(x)),
\end{equation*}
where
\[
M(u_0, \ldots, u_s)(x)=\left( \begin{array}{ccc}
u_0(x)& \ldots& u_s(x)\\
u'_0(x)& \ldots & u'_s(x)\\
\vdots& & \vdots\\
u_0^{(s)}(x) &&  u_s^{(s)}(x) 
\end{array}\right).
\]
We say that $\CF$ is an {\it Extended Complete Chebyshev} system or an ECT-system on a closed interval $[a,b]$ if and only if for any $k,$ $0\leq k\leq n,$ $[u_0,u_1,\ldots,u_k]$ is an ET-system. In order to prove that $\CF$ is an ECT-system on $[a,b]$ it is sufficient and necessary to show that $W(u_0,u_1,\ldots,u_k)(t)\neq 0$ in $[a,b]$ for $0\leq k\leq n,$ see \cite{KASTU1966}. 

\subsection{Preliminary Results} In this section, we introduce some results regarding extended Chebyshev system.

A first classical result is the following:
\begin{theorem}[\cite{KASTU1966}]\label{t1}
Let $\F=[u_0, u_1, \ldots, u_n]$ be an ECT-system on a closed interval  $[a, b]$. Then, the number of isolated zeros for every element of $\mbox{Span}(\F)$ does not exceed $n$. Moreover, for each configuration of $m \leq n$ zeros, taking into account their multiplicity, there exists $F\in \mbox{Span}(\F)$ with this configuration of zeros.
\end{theorem}

Next results, proven in  \cite{NOTOR2017}, extend the above theorem when some of the Wronskians vanish.

\begin{theorem}[\cite{NOTOR2017}]\label{t2}
	Let $\F=[u_0, u_1, \ldots u_n]$ be an ordered set of functions on $[a, b]$. Assume that all the Wronskians $W_i(x)$, $i\in\{0, \ldots, n-1\}$, are nonvanishing except $W_n(x)$, which has exactly one zero on $(a, b)$ and this zero is simple. Then, the number of isolated zeros for every element of $Span(\F)$ does not exceed $n+1$. Moreover, for any configuration of $m \leq n+1$ zeros there exists $F\in Span(\F)$ realizing it.
\end{theorem}

\begin{theorem}[\cite{NOTOR2017}]\label{t3} 
 Let  $\mathcal{F}=\left[u_{0}, u_{1}, \ldots, u_{n}\right]$   be an ordered set of analytic functions in $[a, b].$  Assume that all the $\nu_{i}$  zeros of the Wronskian  $W_{i}$   are simple for $i\in\{0, \ldots, n\}.$   Then the number of isolated zeros for every element  of  $\operatorname{Span}(\mathcal{F})$  does not exceed 
\begin{equation}\label{upper}
n+\nu_{n}+\nu_{n-1}+2\left(\nu_{n-2}+\cdots+\nu_{0}\right)+\mu_{n-1}+\cdots+\mu_{3}
\end{equation}
where $\mu_i=\min(2 \nu_i, \nu_{i-3}+\ldots+\nu_0)$, for $i\in\{3, \ldots, n-1\}.$
\end{theorem} 

\begin{remark}\label{remark}
In Theorem \ref{t3}, we are assuming that all the zeros of the Wronskians $W_{i},$ $i\in\{0, \ldots, n\},$ are simple. This condition can be dropped as follows:

\medskip

Assume that, for each $i\in\{0, \ldots, n\},$ the Wronskian  $W_{i}$ has $\nu_{i}$ zeros counting multiplicity. Then, the number of simple zeros for every element  of  $\operatorname{Span}(\mathcal{F})$  does not exceed \eqref{upper}. 

\medskip

Indeed, if there exists an element 
$f=\sum_{i=0}^n a_i u_i\in\operatorname{Span}(\mathcal{F})$ 
for which the number of simple zeros exceeds \eqref{upper}, then by perturbing the functions $u_i,$ let us say $u_i^{\e},$ for $i\in\{0, \ldots, n\},$ the function
$f_{\e}=\sum_{i=0}^n a_i u^{\e}_i$ 
would still exceed \eqref{upper}, because we are assuming that the zeros of $f$ are simple. In addition, such a perturbation can be chosen in such a way that each Wronskian $W^{\e}_{i}$  of ordered set of functions $\left[u_{0}^{\e}, u_{1}^{\e}, \ldots, u_{i}^{\e}\right],$ for $i\in\{0, \ldots, n\},$ has less than or exactly $\nu_{i}$  zeros, all of them simple. This contradicts Theorem \ref{t3}.
\end{remark}

\subsection{New families of ET-systems with accuracy}

In what follows, for $k\in\mathbb{Z_+}$ and $\lambda \in \R,$ we consider the functions $u_1^k,\ldots,u_{23}^k$, and $u_{24}^{k,\lambda}$ defined on $(0,\infty)$ as

\[
\begin{array}{ll}
u_{1}^k(x)=1, &  u_2^k(x)=x,\\
  u_3^k(x)= x^{2 k-2},& u_4^k(x)=x^{2k},\\
  u_5^k(x)= x^{2k+1},&  u_6^k(x)= x^{4 k-2},\\
 u_7^k(x)=x^{4 k},&u_8^k(x)=x^{4 k+1},\\
  u_9^k(x)=  x^{6 k-2},&u_{10}^k(x)= x^{6 k},\\
 u_{11}^k(x)=x^{6 k+1}, & u_{12}^k(x)=x(1+ x^{4k}),\\
 u_{13}^k(x)=x^{4k}+x^2,&u_{14}^k(x)=x+(2 k+1) x^{8 k+1}\\
   u_{15}^k(x)= \left(x^{4 k}+x^2\right) \tan ^{-1}\left(x^{2 k-1}\right),&  u_{16}^k(x)= \left(x^{4 k-2}+1\right) \left(2 k x^{4 k-1}+x\right), \\
 u_{17}^k(x)= \left(x^{4 k-2}+1\right) \left(2 k x^{4 k-1}+x\right) \tan ^{-1}\left(x^{2 k-1}\right),&u_{18}^k(x)=x^{1/k} \left((2 k+1) x^2+1\right)^3, \\
 u_{19}^k(x)=-x^{1/k} \left((2 k+1) x^3+x\right)^2, &u_{20}^k(x)=-x^{\frac{1}{k}+3} \left((2 k+1) x^2+1\right)^2, \\ u_{21}^k(x)=x^{\frac{3}{2 k}+1} \left((2 k+1) x^2+1\right)^3,&
u_{22}^k(x)=x^{\frac{1}{k}+1} \left((2 k+1) x^2+1\right)^3, \\
  u_{23}^k(x)=\left(x^2+1\right) x^{\frac{3}{2 k}} \left((2 k+1) x^2+1\right)^3, \end{array}
\]
and
\[
\begin{array}{rl}
u_{24}^{k, \lambda}(x)=&x^5 \lambda^3(2 k  +1)^3+x^2 \left(3 \left(8 k^2+6 k+1\right) \lambda ^2+1\right)+\lambda  x \left(-4 k^2 \lambda ^2-2 k \left(\lambda ^2-3\right)+3\right)\\
& +1+(2 k+1)( \lambda  x^3 \left((4 k^2 +1)\lambda ^2+k \left(4 \lambda ^2-6\right)+3\right)+ x^4 \left(3 \lambda ^2+k \left(6 \lambda ^2+2\right)\right)).\\
\end{array}
\]

We define on $(0, \infty)$ the  ordered set of functions 
\[
\begin{array}{l}
\F_1^k=[u_{1}^k, u_{12}^k, u_4^k],\\
\F_2^k=[u_{13}^k,u_{15}^k, u_5^k, u_2^k],\\
\F_3^k=[u_{1}^k, u_4^k, u_9^k,  u_{16}^k, u_{17}^k],\\
\F^k_4=[ u_4^k, u_9^k, u_6^k, u_3^k, u_{16}^k, u_{17}^k],\\
\F^k_5=[u_{1}^k, u_4^k, u_7^k, u_8^k, u_{10}^k, u_5^k, u_{11}^k, u_{14}^k],\\
\F^k_6=[u_{1}^k, u_4^k, u_9^k, u_6^k, u_3^k, u_{16}^k, u_{17}^k]  \text{ and}\\
\F^{k, \lambda}_7=[u_{18}^k, u_{19}^k, u_{20}^k, u_{21}^k, u_{22}^k, u_{23}^k, u_{24}^{k, \lambda}]. 
\end{array}
\]

\begin{proposition} \label{pro1}  The sets of functions $\F_2^1,$ $\F_3^1,$ $\F_4^2,$ and $\F_5^k,$ for $k\geq 1$ are ECT-systems on $[a,b]$, for any $0<a<b$.
	\end{proposition}

\begin{proof}

It is enough to show that the Wronskians defined by $\F_2^1,$ $\F_3^1,$ $\F_4^2,$ and $\F_5^k,$ $k\geq 1$, do not vanish in $(0, \infty)$, which, by definition, implies that all of these sets are ECT-systems.
	
\bigskip
	
The Wronskians of the family $\F_2^1$ are provided by
	\[
	\begin{array}{rcl}
	W_0(x) & = & x^2+x^{4},\\
	W_1(x) & = & x^2 \left(x^4+x^2\right),\\
	W_2(x) & = &-\dfrac{4 x^9}{x^4+x^2}, \\
	W_3(x) & = & \dfrac{32 x^9}{(x^2 + x^4)^2},
	\end{array}
	\]	
	which, clearly, do not vanish in $(0, \infty)$.  
	
	\bigskip
	
The Wronskians of $\F_3^1$  are provided by
	\[
	\begin{array}{rcl}
	W_0(x)&=&1,\\
	W_1(x)&=& 2 x,\\
	W_2(x)&=& 16 x^{3},\\
	W_3(x)&=&48x(1 - 3 x^2 + 10 x^4),\\
	W_4(x)&=&\dfrac{1536 x^{3} (9 + 2 x^2)}{(1 + x^2)^3},
	\end{array}
	\]
which, clearly, do not vanish in $(0, \infty)$.  

\bigskip

The Wronskians of the family $\F_4^2$ are provided by
	\[
	\begin{array}{rcl}
	W_0(x)&=&x^{4},\\
	W_1(x)&=& 6 x^{13},\\
	W_2(x)&=& -48 x^{17},\\
	W_3(x)&=&3072x^{16},\\
	W_4(x)&=&27648 x^{13} \left(924 x^{12}-25 x^6+15\right),\\
	W_5(x)&=& \dfrac{47775744 x^{24} \left(2464 x^{18}+42156 x^{12}+3975 x^6+3325\right)}{\left(x^6+1\right)^4}.
	\end{array}
	\]
which, clearly, do not vanish in $(0, \infty)$.  

\bigskip

The Wronskians of $\F_5^k$ are
	\[
	\begin{array}{rcl}
	W_0(x)&= &1,\\
	W_1(x)&= &2 k x^{ 2 k-1},\\
	W_2(x)&= &16 k^3 x^{6 k-3},\\
	W_3(x)&= &16 k^3 \left(8 k^2+6 k+1\right) x^{10 k-5},\\
	W_4(x)&= &768 k^6 (2 k-1) \left(8 k^2+6 k+1\right) x^{16 k-9},\\
	W_5(x)&= &-1536 k^7 \left(1-4 k^2\right)^2 \left(16 k^2-1\right) x^{18 k-13},\\
	W_6(x)&= &-12288 k^9 (2 k+1)^3 (4 k-1) (6 k+1) \left(-8 k^2+2 k+1\right)^2 x^{24 k-18},\\
	W_7(x)&= &-589824 k^{12} (2 k+1)^3 (4 k-1) (6 k+1) \left(-8 k^2+2 k+1\right)^2 x^{24 (k-1)} \\
& &	\left(48 k^3-44 k^2+12 k-1+(2 k+1)^2 (4 k+1) (6 k+1) (8 k+1) x^{8 k}\right).
	\end{array}
	\]	
	It can easily be seen that, for $k\in\mathbb{Z}_+$, the Wronskians do not vanish in $(0, \infty).$
	
	This ends the proof of Proposition \ref{pro1}.
	\end{proof}

	\begin{proposition} \label{pro2} 
	The sets of functions $\F_1^k,$ for $k\geq 1,$  $\F_2^k,$  for $k\geq2$,  $\F_4^k$, for $k>2$, and $\F_6^2$ are  ET-system with accuracy $1$ on $[a,b]$, for any $0<a<b$.
	\end{proposition}
\begin{proof}

For each set $\F_1^k,$ for $k\geq 1,$ and $\F_2^k,$ $\F_4^k,$ and $\F_6^k$, for $k\geq2$,  we will show that all their Wronskians are nonvanishing  except the last, which has exactly one simple zero in $(0, \infty)$. Thus, from Theorem \ref{t2},  we will have that all of these sets are ET-systems with accuracy 1.

\bigskip

The Wronskians of the family $\F_1^k$ are provided by
	\[
		\begin{array}{rcl}
			W_0(x)&=&1,\\
			W_1(x)&=& (4 k+1) x^{4 k}+1,\\
			W_2(x)&=& 2 k x^{2 (k-1)} ( - (1 + 6 k + 8 k^2) x^{4 k} + 2 k-1).
		\end{array}
	\]
	It can easily be seen that, for $k\in\mathbb{Z}_+$,  the Wronskians $W_0(x)$ and $W_1(x)$ do not vanish in $\R$ and $W_2(x)$ has exactly one positive zero, which is simple. 

\bigskip

The Wronskians of the family $\F_2^k$ are provided by
	\[
	\begin{array}{rcl}
	W_0(x) & = & x^2+x^{4k},\\
	W_1(x) & = & (2 k-1)(x^{2k+2}+x^{6k}),\\
	W_2(x) & = & -\dfrac{4(2k-1)^3x^{8k+1}}{x^{2 }+x^{4 k}},\\ 
	W_3(x) & = &-\dfrac{16 k (2 k-1)^3 x^{8 k-3} \left((k-1) (4 k-1) x^{4 k-2}+1-3 k \right)}{\left(x^{4 k-2}+1\right)^2}.
	\end{array}
	\]	
Again, It can easily be seen that, for $k\in\mathbb{Z}_+$ such that $k\geq 2$,  the Wronskians $W_0(x), W_1(x), W_2(x)$ do not vanish in $(0, \infty)$ and $W_3(x)$ has a unique positive zero, which is simple. 

\bigskip

The Wronskians of the family $\F_4^k$, for $k>2$, are provided by
	\[
		\begin{array}{rcl}
			W_0(x)&=&x^{2k},\\
			W_1(x)&=& (4 k-2) x^{8 k-3},\\
			W_2(x)&=& -8 (k-1) k (2 k-1) x^{12 k-7},\\
			W_3(x)&=&128 (k-1) k^3 (2 k-1) x^{14 k-12},\\
			W_4(x)&=&128 (k-1) k^3 (2 k-1)^3 x^{14 k-15} P_{0,k}\left(x^{4 k-2}\right),\\
			W_5(x)&=& \dfrac{8192 (1-2 k)^6 (k-1) k^3 x^{24 k-16}}{\left(x^{4 k}+x^2\right)^4}P_{1,k}(x^{4k-2}).
		\end{array}
	\]
	where
	\[
	\begin{array}{ll}
	P_{0,k}(x)=&6 k (4 k-1) (6 k-1) x^2-(2 k+1)^2x+3 (2 k (4 k-9)+9),\\
		\end{array}
	\]
	and
		\[
	\begin{array}{ll}
	P_{1,k}(x)=& -4 (k-1) k (2 k-5) (3 k-2) (4 k-1) (6 k-1) x^3\\
	&+4 (k (k (k (4 k (9 k (8 k+3)-281)+949)-249)+20)-1) x^2\\
	&+(3 k-1) (4 k (k (4 k (36 k-89)+185)+5)-29) x\\
	&+(2 k-3) (3 k-1) (4 k-3) (4 k-1) (10 k-1).\\
	\end{array}
	\]
	Notice that the Wronskians $W_i(x) \neq 0$ for $i=0,1,2,3$ do not vanish in $(0, \infty)$. Next, we shall show that $W_4(x)>0$ in $(0, \infty)$ and $W_5(x)$  has one positive zero, which is simple.  By computing the discriminant of $P_{0,k}$ and $P_{1,k}$ we obtain
\[\text{Dis}(P_{0,k})=-13824 k^5+36880 k^4-29056 k^3+7800 k^2-640 k+1,\]
and
\[
\begin{array}{ll}
\text{Dis}(P_{1,k})=&-16 (2 k-1)^6 (3 k-1) \left(576 k^6-720 k^5+380 k^4-212 k^3+183 k^2-89 k+17\right)\\
& \left(41-12428 k-51458 k^2+3664611 k^3-32461588 k^4+126891032 k^5-\right.\\
&\left.257528192 k^6 + 276914736 k^7 - 143578944 k^8 + 22830336 k^9 + 
 3317760 k^{10}\right).
\end{array}
\]
Performing a simple analysis, it can be seen that $\text{Dis}(P_{0,k}), \text{Dis}(P_{1,k})<0$ for $k>2$. Therefore, $P_{0,k}(x)$ does not admit real zeros and $P_{1,k}(x)$ has at most one real zero, counting multiplicity. Consequently  $W_4(x)$ does not vanish in $\R$ and  $W_5(x)$ has at most one positive zero, which is simple if it exists. Now,
$P_{1,k}(0) =(-3 + 2 k) (-1 + 3 k) (-3 + 4 k) (-1 + 4 k) (-1 + 10 k) > 0$ and
\[
\displaystyle\lim_{x\to \infty}\text{Sign}(P_{1,k}(x))= \text{Sign}\left( \left(40 k - 516 k^2 + 2220 k^3 - 3808 k^4 + 2640 k^5 - 576 k^6\right)\right)\ <0.
\]
Therefore, $W_5(x)$ has exactly one positive zero, which is simple.

\bigskip

The Wronskians of the family $\F_6^2$ are
	\[
		\begin{array}{rcl}
			W_0(x)&=&1,\\
			W_1(x)&=& 4 x^{3},\\
			W_2(x)&=& 240 x^{11},\\
			W_3(x)&=&-11520 x^{14},\\
			W_4(x)&=&1474560 x^{12},\\
			W_5(x)&=&13271040 x^8 P_{2,2}(x^6),\\
			W_6(x)&=& -\dfrac{183458856960 x^{18}P_{4,2}(x^6) }{\left(x^6+1\right)^5},
		\end{array}
	\]
where
	\[
	\begin{array}{ll}
	P_{2,2}(x)=15 - 175 x + 12012 x^2\\
		\end{array}
	\]
	and
		\[
	\begin{array}{ll}
	P_{4,2}(x)=8008 x^{4}+460390 x^{3}-993711 x^{2}+29800 x-6650.\\
	\end{array}
	\]
Clearly, $W_i(x) \neq 0$ for $i=0,1,2,3,4,5$. Now, the discriminant  of $P_{4,2}(x)$ is given by $$\text{Dis}(P_{4,2})=-5822536650303705842827108906279200.$$ Thus, $P_{4,2}(x)$ has at most two real zeros counting multiplicity. Additionally, $P_{4,2}(0)=-6650$ and $ \lim_{x\to \infty} P_{4,2}(x)=\infty$. Therefore, $P_{4,2}(x)$ and, consequently, $W_{6}(x)$ have exactly on positive zero, which is simple. 
 
This ends the proof of the Proposition \ref{pro2}.

\end{proof}

	\begin{proposition} \label{pro3} 
	The sets of functions $\F_6^k$, for $k>2$, is an ET-system with accuracy $1$ on $[a,b]$, for any $0<a<b$.
	\end{proposition}

\begin{proof}

Let $\Ge^k=[u_{0}^k, u_4^k, u_9^k, u_6^k, u_3^k, u_{16}^k]$  and $\He^k_{\alpha, \beta}=[ u_4^k, u_9^k, u_6^k, u_3^k, \alpha u_{0}^k+ \beta u_{16}^k+  u_{17}^k]$ be  ordered sets. Observe that  
$$\text{Span}(\F_6^k)=\text{Span}(\Ge^k) \cup \bigcup_{\alpha, \beta\in \mathbb{R}} \text{Span}(\He^k_{\alpha, \beta}).$$
 The demonstration of this lemma will be done in two steps. Firstly, we will show that the Wronskians defined by $\F_6^k$ are nonvanishing  except for the last one, which has two simple zeros. Secondly, we will prove that $\Ge^k$ is an ECT-system and that the Wronskians defined by $\He^k_{\alpha, \beta}$ are nonvanishing  except the last one, which has at most 3 zeros, counting multiplicity. Thus, from Theorems \ref{t1}, \ref{t2}, \ref{t3}, and
Remark \ref{remark}, we have that $7\leq\z(\F^6_k)\leq 8$, $\z(\Ge^k)=5$ and $4\leq \z(\He^k_{\alpha, \beta})\leq 7$. Hence, we conclude that $\z(\F^6_k)= 7.$

The Wronskians of the family $\F_6^k$ are provided by
	\[
		\begin{array}{rcl}
			W_0(x)&=&1,\\
			W_1(x)&=& 2 k x^{2 k-1},\\
			W_2(x)&=& 8 k (k (6 k-5)+1) x^{8 k-5},\\
			W_3(x)&=&-64 (1-2 k)^2 (k-1) k^2 (3 k-1) x^{12 k-10},\\
			W_4(x)&=&2048 k^4 (3 k-1) \left(2 k^2-3 k+1\right)^2 x^{14 k-16},\\
			W_5(x)&=&2048 (1-2 k)^4 (k-1)^2 k^4 (3 k-1) x^{14 k-20} P_{2,k}\left(x^{4 k-2}\right),\\
			W_6(x)&=& \dfrac{262144 (k-1)^2 k^4 (2 k-1)^7 (3 k-1) x^{24 k-20} P_{4,k}(x^{4k-2})}{\left(x^{4 k}+x^2\right)^5},
		\end{array}
	\]
where
	\[
	\begin{array}{ll}
	P_{2,k}(x)=6 k (4 k-1) (6 k-1) (8 k-3) x^2-(4 k-1) (2 k+1)^2 x+3 (2 k (4 k-9)+9)\\
		\end{array}
	\]
	and
		\[
	\begin{array}{ll}
	P_{4,k}(x)=\!\!\!\!\!\! &4 (k-1)^2 k (2 k-5) (3 k-2) (4 k-1) (6 k-1) (8 k-3) x^4\\
	&-2 (3 k-2) (4 k-1) (k (4 k (k (4 k (2 k (78 k-179)+235)-89)-59)+35)-1) x^3\\
	&+(3 k-2) (k (4 k (2 k (10 k (k (4 k (48 k-61)+177)-183)+1017)-465)+201)-19) x^2\\
	&-4 (3 k-1) (5 k-2) (2 k (k (4 k (k (4 k-19)+44)-75)-19)+13) x\\
	&+(2 k-3) (3 k-1) (4 k-3) (4 k-1) (5 k-2) (10 k-1).\\
	\end{array}
	\]
Clearly, $W_i(x) \neq 0$ for $i=0,1,2,3,4$. Now, we show that, for $k>2,$ $W_5(x)>0$ in $(0, \infty)$ and $W_6(x)$  has two positive zeros, which are simple.  
By computing the discriminant of $P_{2,k}$ and $P_{4,k}$ we obtain
\[\text{Dis}(P_{2,k})=-( 4 k-1) A_k,\]
where 
\[
A_k= 1 + 1948 k - 20744 k^2 + 66464 k^3 - 77296 k^4 + 27584 k^5
\]
and
\[
\begin{array}{ll}
\text{Dis}(P_{4,k})=-192 (2-3 k)^2 (1-2 k)^{12} (3 k-1) (4 k-1) (5 k-2) B_k C_k,
\end{array}
\]
with
\[
\begin{array}{ll}
B_k=&206 - 1917 k + 5508 k^2 + 14166 k^3 - 161955 k^4 + 507294 k^5 - 
336876 k^6 - 2819520 k^7 \\
&+ 11872944 k^8 - 24994208 k^9 + 
32211648 k^{10} - 24318720 k^{11} + 8294400 k^{12},\\
C_k=&1234 + 1406151 k - 140801881 k^2 + 1655961863 k^3 + 15757275163 k^4 \\
&- 454467427122 k^5 + 3991908595280 k^6 -18758368588312 k^7 + 52157245218176 k^8 \\
&- 84657031448672 k^9 + 65764683807488 k^{10} +13116254256768 k^{11}\\
& - 75206228610816 k^{12}+ 66368938080256 k^{13}+ 1454789099520 k^{17}.
\end{array}
\]
It is straightforward to see that $A_k, B_k, C_k>0$. Thus, we get that $\text{Dis}(P_{4,k}), \text{Dis}(P_{2,k})<0$.  Therefore, $P_{2,k}(x)$ and, consequently, $W_ 5(x)$, do not admit real zeros. Additionally, $P_{4,k}(x)$ and, consequently, $W_6(x)$ have at most two positive zeros counting multiplicity. Furthermore, 
\[
\begin{array}{rcl}
P_{4,k}(0) &=&( 2 k-3) ( 3 k-1) ( 4 k-3) ( 4 k-1) ( 5 k-2) ( 10 k-1),\\
P_{4,k}(2)&= &-6 - 3687 k + 63459 k^2 - 351684 k^3 + 787140 k^4 - 528768 k^5\\ 
& &-  478272 k^6 + 738816 k^7 - 221184 k^8,\\
\end{array} 
\] 
and $\displaystyle\lim_{x\to \infty} P_{4,k}(x)=\infty.$
Since $\sgn(P_{4,k}(0))=-\sgn(P_{4,k}(2))=1$, it follows that $P_{4,k}(x)$ and, consequently, $W_6(x)$ have exactly two positive zeros, which are simple. Therefore, from Theorems  \ref{t2} and \ref{t3}, it follows that 	$7\leq \z(\F_6^k)\leq 8$.

Since the Wronskians of $\Ge^k$ are, clearly, equal to the first six   Wronskians of $\F_6^k$, we get that $\z(\Ge^k)$ is an ECT-system and, from Theorem  \ref{t1}, $\z(\Ge^k)=5.$

Now, computing  the Wronskians of $\He^k_{\alpha, \beta},$ we obtain
\[
\begin{array}{rcl}
W_0(x)&=&x^{2k},\\
W_1(x)&=& (4 k-2) x^{8 k-3},\\
W_2(x)&=& -8 (k-1) k (2 k-1) x^{12 k-7},\\
W_3(x)&=&128 (k-1) k^3 (2 k-1) x^{14 k-12},\\
W_4(x)&=&128 (2 k-1)^3 (k-1) k^3 x^{14 k-15} S_3^k\left(x^{4 k-2}\right)  Q_{\alpha,\beta}^k(x),\\
\end{array}
\]
where

\[
\begin{array}{rcl}
Q_{\alpha,\beta}^k(x) & = & \left(\dfrac{\alpha S_1^k(x)+ (2 k-1) x^{2 k+9} S_2^k\left(x^{4 n-2}\right)}{(2 k-1) x^4 \left(x^{4 k}+x^2\right)^3 S_3^k\left(x^{4 k-2}\right)}+  \tan ^{-1}\left(x^{2 k-1}\right)+\beta \right),\\
S_{1}^k(x) & = & 16 (-1 + k) k (-1 + 3 k) x^3 (x^2 + x^{4 k})^3,\\
S_{2}^k(x) & = & -3 (9 + 2 k (-9 + 4 k)) + (-71 + 4 (37 - 15 k) k) x\\
& & + (-1 + 2 k) (61 + 2 k (-37 + 92 k)) x^2 \\
& &+ (-1 + 
4 k (3 + k (-41 + 96 k))) x^3 + 6 k (-1 + 4 k) (-1 + 6 k) x^4,\\
S^k_3(x) & = & 3 (9 + 2 k (-9 + 4 k)) - (1 + 2 k)^2 x + 6 k (-1 + 4 k) (-1 + 6 k) x^2.
\end{array}
\]
Clearly, $W_i(x) \neq 0$ for $i=0,1,2,3$.  The derivative of $Q_{\alpha,\beta}^k(x)$ can be written  as
 \[(Q^k_{\alpha, \beta})'(x)=R^k(x)S_{\alpha}^k(x),\]
 with
  \[
  \begin{array}{rcl}
 R^k(x)&=&\dfrac{ -16 (k-1) k (3 k-1) x^2  q_1^k(x^{4k-2})}{ (2 k-1) x^8  S^k_{3}\left(x^{4 k-2}\right)^2},\\
 S_{\alpha}^k(x)&=& \alpha+ \dfrac{ 4 (1 - 2 k)^2 x^{  10 k+4} q_2^k(x^{4k-2})}{ (k-1) k (3 k-1)  \left(x^{4 k}+x^2\right)^4 q_{1}^k(x^{4k-2})},
  \end{array}
  \] 
 where
 \[
  \begin{array}{rcl}
  q_1^k(x)&=&6 k (4 k-1) (6 k-1) (8 k-3) x^2-(2 k+1)^2 (4 k-1) x+3 (2 k (4 k-9)+9),\\ 
  q_2^k(x)&=& 9 - 171 k + 1052 k^2 - 2692 k^3 + 2816 k^4 - 960 k^5 \\
  & &+ (-29 + 107 k + 680 k^2 - 3644 k^3 + 4848 k^4 - 1728 k^5) x 
  \\
  & &+(4 - 80 k + 996 k^2 - 3796 k^3 + 4496 k^4 - 432 k^5 - 1152 k^6) x^2\\
  & &+(-40 k + 516 k^2 - 2220 k^3 + 3808 k^4 - 2640 k^5 + 576 k^6) x^3.
  \end{array}
  \]
Observe that, for $k>2,$ the function $q_1^k(x)$ is positive. Indeed, 
  \[\text{Dis}(q_1^k)=-(4 k-1) (1 + 1948 k - 20744 k^2 + 66464 k^3 - 77296 k^4 + 27584 k^5)<0\]
  and $q_1^k(0)>0 $.  Moreover, notice that $R^k(x)$ does not vanish in $(0, \infty)$ and
\[(S^k_{\alpha})'(x)=\dfrac{ 8 (1-2 k)^2 x^{10 k-5} q_3(x^{4 k - 2}) q_4^k(x^{4 k - 2})}{( k-1) k (3 k-1) (1+ x^{4 k-2})^5 ( q_1^k(x^{4 k - 2}))^2 },\]
where 
\[
\begin{array}{rcl}
q_3^k(x)&=&-27 - 6 k (-9 + 4 k) + (1 + 4 k (1 + k)) x - 
6 k (-1 + 4 k) (-1 + 6 k) x^2,\\
q_4^k(x)&=& (-3 + 2 k) (-1 + 3 k) (-3 + 4 k) (-1 + 4 k) (-2 + 5 k) (-1 + 10 k)- \\
& &4 (-1 + 3 k) (-2 + 5 k) (13 + 
2 k (-19 + k (-75 + 4 k (44 + k (-19 + 4 k))))) x \\
& & +(38 - 459 k + 4323 k^2 - 21852 k^3 + 53688 k^4 - 72240 k^5)x^2\\ 
& &+ (81520 k^6 - 89280 k^7 + 46080 k^8) x^2+ (4 - 162 k + 1738 k^2 - 4608 k^3 )x^3
\\
& &+(-17208 k^4 + 114176 k^5 - 226240 k^6 + 192384 k^7 - 59904 k^8)x^3\\
& & +(4 (-1 + k)^2 k (-5 + 2 k) (-2 + 3 k) (-1 + 4 k) (-1 + 6 k) (-3 + 8 k)) x^4.\\
\end{array}
\]
By computing the discriminant of $q_3^k$ and $q_4^k$ we obtain
\[\text{Dis}(q_3^k)=1 - 8 k (80 + k (-975 + 2 k (1816 + k (-2305 + 864 k))))\]
and
\[\text{Dis}(q_4^k)= -192 (3 k-2)^2 ( 2 k-1)^{12} ( 3 k-1) ( 4 k-1) ( 5 k-2) D_k E_k,
\]with
\[
\begin{array}{rcl}
D_k&=&206 - 1917 k + 5508 k^2 + 14166 k^3 - 161955 k^4 + 507294 k^5 - 
336876 k^6 \\
& &-2819520 k^7 + 11872944 k^8 - 24994208 k^9 + 32211648 k^{10} - 24318720 k^{11} \\
& &+ 8294400 k^{12},\\
E_k&=&1234 + 1406151 k - 140801881 k^2 + 1655961863 k^3 + 15757275163 k^4 \\
& &- 454467427122 k^5+3991908595280 k^6 - 18758368588312 k^7 + 52157245218176 k^8 \\
& &- 84657031448672 k^9 +65764683807488 k^{10} + 13116254256768 k^{11} \\
& &- 75206228610816 k^{12} + 66368938080256 k^{13}-30092670877696 k^{14} \\ & &+12225870102528 k^{15} - 5928649555968 k^{16} + 1454789099520 k^{17}.\\
\end{array}
\]
Thus, by straightforward  computations, we obtain  $\text{Dis}(q_3^k), \text{Dis}(q_4^k)<0$, for $k>2$. Therefore, $q_{3}(x)$  does not admit real zeros and $q_4^k$ has at most two positive zeros counting multiplicity. It implies that  the number of zeros of $(S^k_{\alpha})'(x)$ counting multiplicity is at most two. Consequently, $(Q^k_{\alpha, \beta})'(x)$ has at most 3 zeros. Notice that
\[\lim_{x\to 0} \text{Sign}((Q^k_{\alpha, \beta})'(x))=\lim_{x \to \infty }\text{Sign}((Q^k_{\alpha, \beta})'(x))=-\alpha.\]
 For $\alpha\neq 0$, follows that $(Q^k_{\alpha, \beta})'(x)$ has at most 2 zeros. Therefore,  $Q^k_{\alpha, \beta}$ and, consequently, $W_4$ have at most 3 positive zeros.  Thus, from Theorem  \ref{t3} and
Remark \ref{remark}, we get that  $ \z(\He^k_{\alpha, \beta})\leq 7$. For $\alpha= 0$, it follows that $\text{Span}(\He^k_{\alpha, \beta})\subset\text{Span} (\F_4^k)$. Taking Proposition \ref{pro2} into account, we get that  $ \z(\He^k_{\alpha, \beta})= 6$. This ends the proof of the Proposition \ref{pro3}.
 
\end{proof}

\begin{proposition} \label{pro4} 
	For $\lambda\in\R$,  $ \z(\F_7^{1, \lambda})\leq 10$ on $[a,b]$, for any $0<a<b.$ In addition, for $\la=2,$ there exists a function in $\text{Span} (\F_7^{1, 2})$ having 8 simple zeros in $(0, \infty)$.
\end{proposition}

\begin{proof}
 Let 
 \[f(x)= a_0 u_{18}^1(x)+a_1 u_{19}^1(x)+a_2 u_{20}^1(x)+a_3 u_{21}^1(x)+a_4 u_{22}^1(x)+a_5 u_{23}^1(x)+a_6 u_{24}^{1, \lambda}(x)\]
 be a function in $\text{Span}(\F_7^{1,\lambda})$. The $5$th derivative of $f$, $f^{(5)}(x),$ is written as a linear combination of the functions of the ordered set $$J_0=[1, x, x^2, x^3,(u_{21}^1)^{(5)}(x),(u_{23}^1)^{(5)}(x)].$$ Computing  the Wronskians of $J_0,$ we get
\[
\begin{array}{rcl}
W_0(x)&=&1,\\
W_1(x)&=&1,\\
W_2(x)&=&2,\\
W_3(x)&=&12,\\
W_4(x)&=& \dfrac{8505 \left(9 \left(429 \left(85 x^4+x^2\right)+35\right) x^2+55\right)}{128 x^{13/2}},\\
W_5(x)&=& \dfrac{120558375}{65536 x^{15}} \left(409280498055 x^{14}+16979438619 x^{12}+2324256363 x^{10}\right.
\\& &\left.+589231071 x^8+64265157 x^6+508833 x^4+23177 x^2+1573\right).
\end{array}
\]
Clearly, all the Wronskian above do not vanish in $(0, \infty)$, which implies that $J_0$ is an ECT-system.  From Theorem \ref{t1},  $f^{(5)}(x)$ has at most 5 zeros and, therefore, $f(x)$ has at most ten zeros. Consequently, $\z(\F_7^{k, \lambda})\leq10.$

Finally, let $f(x) \in \text{Span}(F_7^{1, 2}) $ be provided by
\[
\begin{array}{ll}
f(x)= a_0 u_{18}^1(x)+a_1 u_{19}^1(x)+a_2 u_{20}^1(x)+a_3 u_{21}^1(x)+a_4 u_{22}^1(x)+a_5 u_{23}^1(x)+ u_{24}^{1, 2}(x),
\end{array}
\]
where 
\[\begin{array}{ll}
a_0=-29.674872845038724, &a_1=-88.998921871,\\
a_2=1.777150602939737, &a_3=-2.0194231196937788\times 10^{-5},\\
a_4=0.5926213398946085, &a_5=3.18899089714221\times 10^{-8}.
\end{array}\]
The function $g$, defined by $g(x)=f(x^{2})$, is a polynomial of degree 19 in interval $(0,\infty)$. Direct computation shows that $g$ has 8 zeros, which are simple as $\text{Dis}(g)\neq 0$.
\end{proof}

\begin{proposition} \label{pro5} 
	For $k>1$ and $\la\in\R$,  $ \z(\F_7^{k, \lambda})\leq 14$ on $[a,b]$, for any $0<a<b.$ In addition, for $\la=1,$ there exists a function in $\text{Span} (\F_7^{k, 1})$ having 9 simple zeros in $(0,\infty).$
\end{proposition}
	
	\begin{proof}
Let \[f(x)= a_0 u_{18}^k(x)+a_1 u_{19}^k(x)+a_2 u_{20}^k(x)+a_3 u_{21}^k(x)+a_4 u_{22}^k(x)+a_5 u_{23}^k(x)+a_6 u_{24}^{k, \lambda}(x)\]  be a function in $\text{Span}(\F_7^{k, \lambda})$. Since $(u_{24}^{k,\lambda})^{(8)}=0$ for every $k>1,$ $f^{(8)}(x)$ is written as a linear combination of the functions of the ordered set 
$$\He^k=\left[( u_{18}^k)^{(8)}, ( u_{19}^k)^{(8)}, ( u_{20}^k)^{(8)}, ( u_{21}^k)^{(8)}, ( u_{22}^k)^{(8)}, ( u_{23}^k)^{(8)}\right].$$ 
Computing  the Wronskians of $\He^k,$ we get
\[
	\begin{array}{rcl}
		W_0(x)&=&-\dfrac{(k-1) x^{\frac{1}{k}-8}}{k^8}U_0^k(x),\\
		W_1(x)&=&\dfrac{2 (k-1)^2 (k+1) (2 k-1) (2 k+1) (3 k-1) x^{\frac{2}{k}-15}}{k^{16}}U_1^k(x),\\
		W_2(x)&=&\dfrac{6 (k-1)^2 (k+1)^2 (2 k-1) (2 k+1)^2 \left(9 k^2-1\right) x^{\frac{3}{k}-22}}{k^{24}} U_2^k(x),\\
		W_3(x)&=&\dfrac{81 (2 k-1) (2 k+1)^3 \left(k^2-1\right)^2 \left(18 k^3+27 k^2-2 k-3\right) x^{\frac{9}{2 k}-32}}{1024 k^{35}} U_3^k(x),\\
		W_4(x)&=&-\dfrac{81 (1-2 k)^2 \left(2 k^3+k^2-2 k-1\right)^3 \left(18 k^3+27 k^2-2 k-3\right) x^{\frac{11}{2 k}-43}}{1024 k^{44}} U_4^k(x),\\
		W_5(x)&=& -\dfrac{729 (1-2 k)^2 \left(2 k^3+k^2-2 k-1\right)^3 \left(18 k^3+27 k^2-2 k-3\right) x^{7 \left(\frac{1}{k}-8\right)}}{4194304 k^{56}} U_5^k(x),
	\end{array}
\]
where $U_0, U_1, U_2, U_3, U_4$  and $U_5$ are  polynomials of degrees 6, 8, 12, 18, 22,  and 30, respectively. By straightforward  computations, we get that $U_i(x)$, for $i=0,1,2,3,4$, does not vanish in $(0,\infty)$ and $U_5(x)$ has exactly one positive zero, which is simple.  From Theorem \ref{t2}, it follows that $\z(\He^k) =6$.  Hence, we conclude that $\z(\F_7^{k, \lambda})\leq14.$\\

In what follows, we shall prove that  there exists a function in $\text{Span}(\F_7^{k,1})$ having 9 simple zeros in $(0, \infty)$. Accordingly, let $f(x;a) \in \text{Span}(F_7^{k, 1}) $ be provided by
\[
\begin{array}{ll}
f(x;a)=&(1 + 2 k) (a_0 - 4 (1 + k))u_{19}^k(x)+(-3 a_3 + a_1 (1 + 2 k))(1 + 2 k )u_{20}^k(x)\\
&-4(1+k) u_{18}^k(x)+a_2 u_{21}^k(x)+(-2 a_3 + a_1 (1 + 2 k)) u_{22}^k(x)+a_4 u_{23}^k(x)+u_{24}^{k,1}(x),
\end{array}
\]
where $a=(a_0,a_1,a_2, a_3, a_4) \in \R^5.$ \\

Denote $g_k(x;a):=f(x^{2k};a).$ First, we prove that, for each integer $k>1,$ there exists $\delta_k>0$ such that $g_k(x;a)$ has at least 4 simple zeros in $(0,2)$, for every $a\in B(0, \delta_k).$ Notice that $g(x;0)$ has at least 4 zeros in $(0,2)$ with odd multiplicity, for every $k>1$. Indeed, 
\[g_k(0;0)>0,\,\, g_k(1/2;0)<0,\,\, g_k(1;0)=0,\,\, g_k'(1;0)<0,\,\text{ and }\, g_k(2;0)>0.\]
For $ 2\leq  k\leq 30,$ it is relatively easy to see that $\text{Dis}(g_k(x;0))\neq 0,$ which implies that the 4 zeros above are simple. Now, for $k>30,$
we have
$$g_k(x;0)=H_1(x^{2k})+H_2(x^{2k}),$$
where
\[
\begin{array}{rcl}
H_1(x)&=&\left(8 k^3+6 k+4\right) x^3+2 \left(12 k^2+9 k+2\right) x^2+\left(-4 k^2+4 k+3\right) x+1\\
& &+\left(16 k^2+14 k+3\right) x^4+(2 k+1)^3 x^5,\\
H_2(x)&=&-8 \left(2 k^2+3 k+1\right) x^{\frac{1}{k}+2}-4 (k+1) (2 k+1)^2 x^{\frac{1}{k}+4}-4 (k+1) x^{1/k}.
\end{array}
\]
Notice that $H_1(x)>0$ for $x>0$. The Wronskian of $[H_1(x),H_2(x)]$ can be written as
$W_1(x)=\dfrac{4 (k+1) x^{\frac{1}{k}-1}}{k} P_{5,k}(x)$, with
\[
	\begin{array}{rl}
		P_{5,k}(x)=&-1-\left(4 (k-2) k^2+k+3\right) x+2 \left(k \left(24 k^2+2 k-9\right)-3\right) x^2\\
		&+10 (k-1) (2 k+1) \left(2 k^2+k+1\right) x^3-2 (2 k+1) (k (12 k+19)+6) x^4\\
		&+2 (k-1) (2 k+1)^2 (5 k (2 k+1)+6) x^5-2 (2 k+1)^2 (k (2 k (12 k+7)+15)+5) x^6\\
		&+2 (2 k+1)^3 \left((2 k-5) k^2+3\right) x^7-(2 k+1)^3 (8 k+3) x^8+(k-1) (2 k+1)^5 x^9.
	\end{array}
\]
It is easy to see that $\text{Dis}\left( P_{5,k}^{(3)}\right) <0$ and, since $P_{5,k}^{(3)}(x)$ has degree $6,$ we conclude that $P_{5,k}^{(3)}(x)$ has at most 4 real zeros, counting multiplicity. In addition, $\lim_{x \to \pm\infty }P_{5,k}^{(3)}(x)>0,$ $P_{5,k}^{(3)}(-1/2)<0,$ $P_{5,k}^{(3)}(0)>0,$ and   $P_{5,k}^{(3)}(1/2)<0.$
Thus,  $P_{5,k}^{(3)}(x)$  has two zeros in $(-\infty,0)$ and two zeros in $(0,\infty)$. Therefore, $P_{5,k}^{(2)}(x)$ has at most 3 zeros in $(0,\infty)$, counting multiplicity. Since \[P_{5,k}^{(2)}(0)>0 \text{ and }\lim_{x \to \infty }P_{5,k}^{(2)}(x)>0,\] it follows that $P_{5,k}^{(2)}(x)$ has at most two positive zeros, counting multiplicity. Moreover,\[P_{5,k}(0)<0 \text{ and }\lim_{x \to \infty }P_{5,k}(x)>0.\] Thus,  $P_{5,k}(x)$ has at most 3 zeros, counting multiplicity in $(0,\infty)$.  From Theorem \ref{t3} and Remark \ref{remark}, we get that $g_k(x;0)$ has 4 simple zeros on $(0, \infty)$. Hence,  $g_k(x;0)$ has 4 simple zeros in $(0, 2)$.  

Thus, we have proven that, for each $k>1,$ $g_k(x;0)$ has at least 4 simple zeros in $(0,2).$ Since $g_k(x;a)$ depends continuously on $a$, for each $k>1$ there exists $\delta_k>0$ such that $g_k(x;a)$ has at least 4 simple zeros in $(0,2)$, for every $a\in B(0, \delta_k).$

Now, we prove that, for each integer $k>1,$ there exists $a_k\in B(0, \delta_k)$ such that $g_k(x;a_k)$ has 5 additional simple zeros in $(2,\infty).$
For that, taking $x = y^{-1}$ in $(0, \infty)$, we see that 
$$g_k(y^{-1};a)=\dfrac{1}{y^{3 + 16 k}} h_k(y;a),$$
where $h_k(y;a),$ around $y=0,$ writes
\[
\begin{array}{ll}
h_k(y;a)=& a_4+a_2 y^{2k}+a_3 y^{2k+1}+\dfrac{2 a_4(2+k)}{1+2k} y^{4k}-a_0 y^{4k+1}+\dfrac{3 a_2}{2k+1} y^{6k}+a_1 y^{6k+1}+ y^{6k+3}\\
&+O(y^{6 k+4}).\\
\end{array}
\]
Thus, for each integer $k>1,$ we can choose  $a\in B(0, \delta_k)$ so that $h_k(y;a)$ has $5$ simple positive zeros in a neighbourhood of $y=0.$ Consequently, $g_k(x;a)$ has 5 additional simple positive zeros in a neighbourhood of the infinity. Hence, we found a function in $\F_7^{k,1}$ having at least 9 simple zeros.
\end{proof}

\section{First order analysis}\label{sec:first}
This section is devoted to the proof of statement $(i)$ of Theorem \ref{Theorem-Melnikov2}. In order to apply Theorem \ref{thm:melnikov}, we first write system \eqref{general-system1} in polar coordinates $x= r\cos(\theta)$ and $y=r\sin(\theta),$
\begin{equation}\label{polar-system}
(\dot{r}, \dot{\theta})^T=(0,-1)^T+\displaystyle\sum_{i=1}^6 \varepsilon^i G_i(\theta, r),
\end{equation}
where
\[
G_i(r)=\left\lbrace 
\begin{array}{lll}
(A_i^+(r, \theta), B_i^+(r, \theta))^T,& \text{if}& \sin (\theta)-r^{n-1} \cos^n(\theta)>0,\\
(A_i^-(r, \theta), B_i^-(r, \theta))^T, & \text{if}& \sin (\theta)-r^{n-1} \cos^n(\theta)<0,\\
\end{array}\right. 
\]
with
\[
\begin{array}{l}
A_i^+=\cos (\theta ) (a_{0i}+r (a_{2i}+b_{1i}) \sin (\theta ))+a_{1i} r \cos ^2(\theta )+\sin (\theta ) (b_{0i}+b_{2i} r \sin (\theta )),\\
B_i^+= r^{-1}[-\sin (\theta ) (a_{0i}+a_{2i} r \sin (\theta ))+\cos (\theta ) (r (b_{2i}-a_{1i}) \sin (\theta )+b_{0i})+b_{1i} r \cos ^2(\theta )],\\
A_i^-=\cos (\theta ) (\alpha_{0i}+r (\alpha_{2i}+\beta_{1i}) \sin (\theta ))+\alpha_{1i} r \cos ^2(\theta )+\sin (\theta ) (\beta_{0i}+\beta_{2i} r \sin (\theta )),\\
B_i^-= r^{-1}[-\sin (\theta ) (\alpha_{0i}+\alpha_{2i} r \sin (\theta ))+\cos (\theta ) (r (\beta_{2i}-\alpha_{1i}) \sin (\theta )+\beta_{0i})+\beta_{1i} r \cos ^2(\theta )].
\end{array}
\]
Then, taking $\theta$ as the new time, system \eqref{polar-system} writes as
\begin{equation}\label{polar-system2}
\dfrac{\text{d} r}{\text{d}\theta}=\left\lbrace 
\begin{array}{lll}
\dfrac{\displaystyle\sum_{i=1}^6\varepsilon^i A^+_i(r, \theta)}{-1+\displaystyle\sum_{i=1}^6 \varepsilon^i B_i^+(r, \theta)},& \text{if}& \sin (\theta)-r^{n-1} \cos^n(\theta)>0,\\
\\
\dfrac{\displaystyle\sum_{i=1}^6\varepsilon^i A^-_i(r, \theta)}{-1+\displaystyle\sum_{i=1}^6 \varepsilon^i B_i^-(r, \theta)}, & \text{if}& \sin (\theta)-r^{n-1} \cos^n(\theta)<0.
\end{array}\right. 
\end{equation}
Thus, for $|\varepsilon|\neq 0$  sufficiently small, system \eqref{polar-system2}  and, consequently, system \eqref{polar-system} become equivalent to
\begin{equation} \label{polar-system3}
\dfrac{\text{d} r}{\text{d}\theta}=\left\lbrace 
\begin{array}{lll}
\displaystyle\sum_{i=1}^6\varepsilon^i F_i^+(r, \theta)+ \mathcal{O}(\varepsilon^7), & \text{if}& \sin (\theta)-r^{n-1} \cos^n(\theta)>0,\\
\displaystyle\sum_{i=1}^6\varepsilon^i F_i^-(r, \theta)+ \mathcal{O}(\varepsilon^7), & \text{if}& \sin (\theta)-r^{n-1} \cos^n(\theta)<0,\\
\end{array}\right. 
\end{equation}
where
\[
\begin{array}{l}
F_1^+(r,\theta) = - \cos (\theta ) (a_{01}+r (a_{21}+b_{11}) \sin (\theta ))-a_{11} r \cos ^2(\theta )-\sin (\theta ) (b_{01}+b_{21} r \sin (\theta )),\\
F_1^-(r,\theta) = -\cos (\theta ) (\alpha_{01}+r (\alpha_{21}+\beta_{11}) \sin (\theta ))-\alpha_{11} r \cos ^2(\theta )-\sin (\theta ) (\beta_{01}+\beta_{21} r \sin (\theta )).
\end{array}
\]
 Let $\theta_1(r)=\arctan(r^{n-1})$  be the solution of the equation $\sin \theta-r \cos ^{n-1} \theta=0$ in $[0, \pi/2].$ Thus, for $r>0,$ $\sin \theta-r \cos ^{n-1} \theta<0$  if and only if  $0<\theta<\theta_{1}(r)$ or $\pi-(-1)^n\theta_1(r)<\theta<2 \pi$;  and $\sin \theta-r \cos ^{n-1} \theta>0$  if and only if  $\theta_{1}(r)<\theta< \pi-(-1)^n\theta_{1}(r)$.

According to \eqref{d1}, the first order Melnikov function of system \eqref{polar-system3} is provided by
\begin{equation}
\label{MelFun}
M_1(r)=\int_0^{\theta_1(r)} F_1^-(\theta, r) d \theta+ \int^{\pi-(-1)^{n}\theta_1(r)}_{\theta_1(r)} F_1^+(\theta, r) d \theta+ \int_{\pi-(-1)^{n}\theta_1(r)}^{2\pi} F_1^-(\theta, r) d \theta.
\end{equation}
In order to compute the exact expression of the Melnikov function \eqref{MelFun} we distinguish two  cases, depending on $n$.

\subsection*{Case 1:} Let $n=2k+1$ for a positive integer $k.$ Thus,  
\[
\begin{array}{l}
M_1(r)=\dfrac{1}{2} (v_0 \cos (\theta_1 (r))+r v_1+v_2 \sin (\theta_1(r))),
\end{array}
\]
where
\[
\begin{array}{rcl}
v_0&=&4 \beta_{01}-4 b_{01},\\
v_1&=& -\pi  (a_{11}+ \alpha_{11}+b_{21}+\beta_{21}),\\
v_2&=& 4 (a_{01}-\alpha_{01}).
\end{array}
\]
Notice that the parameter vector $(v_0,v_1,v_2)\in\R^3$ depends on the original parameters in a surjective way. Taking $x=r \cos(\theta_1(r)),$ it follows that
$$x^2+x^{4k+2}=r^2\,\text{ and } \, \sin(\theta_1(r))=\frac{x^{2 k+1}}{r}. $$
Hence,  $M_1(r)=\dfrac{q_1^k(x)}{2 \sqrt{x^{4k}+1}}$, where
\begin{equation*}\label{pol1}
q_1^k(x)=v_1 u_{12}^k(x)+v_2 u_4^k(x)+v_0 u_1^k(x).
\end{equation*}
which belongs to $\mathcal{F}_1^k.$
So, the maximum number of positive zeros of the polynomial function $q_1^k(x)$ coincides with $m_1(2k+1)$.

Note that $q_1^0(x)$ is a first degree polynomial, thus the maximum number of positive simple zeros is $1$. For $k\geq 1$, Proposition \ref{pro2} implies that $\mathcal{F}_1^k$ is an ET-system with accuracy 1 on $[a,b]$ for any $0<a<b.$ Thus, the maximum number of positive simple zeros of $q_1^k(x)$ is 3 and there exists $(v_0,v_1,v_2)\in\R^3$ for which $q_1^k(x)$ has exactly 3 positive simple zeros. Therefore, $m_1(1)=1$ and $m_1(2k+1)=3$ for $k\geq 1$. 

\subsection*{Case 2:}  Let $n=2k$ for a positive integer $k.$ Thus, 
\[
\begin{array}{l}
M_1(r)=r v_0+r v_1 \sin (\theta_1(r)) \cos (\theta_1(r))+r v_2 \theta_1(r)+v_3 \cos (\theta_1(r)),
\end{array}
\]
where
\[
\begin{array}{rcl}
v_0&=&-\dfrac{\pi  (a_{11}+\alpha_{11}+b_{21}+\beta_{21})}{2},\\
v_1&=&a_{11}-\alpha_{11}-b_{21}+\beta_{21},\\
v_2&=& a_{11}-\alpha_{11}+b_{21}-\beta_{21},\\
v_3&=&2 (\beta_{01}-b_{01}).
\end{array}
\]
Notice that the parameter vector $(v_0,v_1,v_2,v_3)\in\R^4$ depends on the original parameters in a surjective way. Again, taking $x=r \cos(\theta_1(r)),$ it follows that $ M_1(r)=\dfrac{q_2^k(x)}{ \sqrt{x^2 + x^{4 n}}}$, where
\begin{equation*}\label{q2}
q_2^k(x)=v_0 u_{13}^k(x)+v_1 u_{5}^k(x)+v_2 u_{15}^k(x)+v_3 u_{2}^k(x),
\end{equation*}
which belongs to $\mathcal{F}_2^k.$ From Proposition \ref{pro1}, $\mathcal{F}_2^1$  is an  ECT-system on $[a,b]$ for any $0<a<b.$ Thus, the maximum number of positive simple zeros of $q_2^1(x)$ is 3 and there exists $(v_0,v_1,v_2,v_3)\in\R^4$ for which $q_2^1(x)$ has exactly 3 positive simple zeros. Therefore, $m_1(2)=3$. For $k\geq 2$, from Proposition \ref{pro2},$\mathcal{F}_2^k$  is an  ET-system with accuracy $1$ on $[a,b]$ for any $0<a<b.$ So, the maximum number of positive simple zeros of $q_2^k(x)$ is $4$ and there exists $(v_0,v_1,v_2,v_3)\in\R^4$ for which $q_2^k(x)$ has exactly $4$ positive simple zeros. Therefore, $m_1(2k)=4$ for $k\geq 2$.

\section{Higher order analysis}\label{sec:higher}

This section is devoted to the proof of statements $(ii)$-$(v)$ of Theorem \ref{Theorem-Melnikov2} for $2\leq l\leq 6$. From Theorem \ref{thm:melnikov},  the simple zeros of the Melnikov function of order $\ell$, $M_{\ell},$ provide periodic solutions of \eqref{polar-system3} whenever $M_i=0$, for $i=1,\ldots,\ell-1$. In our problem, it can be seen that, for each $n\in\mathbb{N}$ and $\ell\in\{2,\ldots,6\}$, there exists $\ell-1$ set of minimal conditions on the parameters of perturbations, $K^n_{\ell,1},\dots,K^n_{\ell,\ell-1}$, such that $M_i(x)=0$ for $i\in\{1,\ldots,\ell-1\}$. In order to obtain $m_{\ell}(n),$ we have to study $M_{\ell}$ for each set of condition. By assuming conditions $K^n_{\ell,i},$ It can be seen that $M_{\ell}=M_{\ell,i}^n,$ where
\[
M_{\ell,i}^n(x)=\dfrac{p_{\ell,i}^n(x)}{q_{\ell,i}^n(x)},
\] 
with $q_{\ell,i}^n(x)\neq 0$ in $(0,\infty)$ and
\begin{table}[h]
	\begin{tabular}{|c|l|l|}
		\hline
		$n=2$& $p_{\ell,i}^n(x)\in \operatorname{Span}(\mathcal{F}_{3}^{1})$ &$ \ell=2,\dots,6\text{ and } i=1,\dots,\ell-1 $\\ \hline
		$n=2k$& $p_{\ell,i}^n(x)\in \operatorname{Span}(\mathcal{F}_{6}^{k})$ & $\ell=2,\dots,6\text{ and } i=1,\dots,\ell-1$\\ \hline
		\multirow{3}{*}{$n=2k+1$} 	&$p_{\ell,i}^n(x)\in \operatorname{Span}(\mathcal{F}_{5}^{k}) $& $\ell=2,\dots,5\text{ and } i=1,\dots,\ell-1$\\ \cline{2-3} 
		&$p_{6,i}^n(x) \in \operatorname{Span}(\mathcal{F}_{5}^{k}) $&$ i=1,\dots,4$\\ \cline{2-3} 
		&$p_{6,5}^n(x)\in \operatorname{Span}(\mathcal{F}_{7}^{k,\lambda})$&  \\ \hline
\end{tabular}
\bigskip
\caption{Structure of the higher order Melnikov functions.}\label{table3}
\end{table}

\subsection*{Case 1: } Let $n=2$ and $\ell\in\{2,\dots, 6\}.$ Assuming conditions on the parameters of perturbations such that $M_i=0$ for $i\in\{1,\ldots,\ell-1\},$ the Melnikov function of order $\ell$ is provided by
\[\label{mi2}
M_\ell(x)=\dfrac{1}{(1+2 x^2)^2}P_\ell(x),
\]
where
\[P_\ell(x)=
C_0^\ell u_1^1(x)+C_1^\ell u_4^1(x)+C_2^\ell u_9^1(x)+C_3^\ell u_{16}^1(x)+C_4^\ell u_{17}^1(x),
\]
which belongs to $\operatorname{Span}(\mathcal{F}_{3}^{1})$ (see Table \ref{table3}). In addition, it can be seen that the parameter vector $(C_0^\ell,\ldots,C_4^{\ell})\in\R^5$ depends on the original coefficients of perturbation in a surjective way. From Proposition \ref{pro1}, $\operatorname{Span}(\mathcal{F}_{3}^{1})$ is an ECT-systems on $[a,b]$ for any $0<a<b.$ Thus, we conclude that the maximum number of positive simple zeros of $P_\ell(x)$ is 4 and there exists $(C_0^\ell,\ldots,C_4^{\ell})\in\R^5$ for which $P_\ell(x)$ has exactly 4 positive simple zeros. Therefore, $m_{\ell}(2)=4$ for ${\ell}=2,\dots,6$.
 
\subsection*{Case 2:} Let $n=2k$, $k>1$, and $\ell\in\{2,\dots, 6\}.$  Assuming conditions on the parameters of perturbations such that $M_i=0$ for $i\in\{1,\ldots,\ell-1\},$ the Melnikov function of order $\ell$ is provided by
\begin{equation*}\label{mi2k}
M_\ell(x)=\dfrac{1}{(1+ 2k x^{4k-2})^2}Q_\ell^k(x),
\end{equation*}
where
 \[Q_\ell(x)=
C_0^\ell u_1^k(x)+C_1^\ell u_4^k(x)+C_2^\ell u_9^k(x)+C_3^\ell u_{6}^k(x)+C_4^\ell u_{3}^k(x)+C_5^\ell u_{16}^k(x)+C_6^\ell u_{17}^k(x),
\]
which belongs to $\operatorname{Span}(\mathcal{F}_{6}^{k})$ (see Table \ref{table3}). In addition, it can be seen that the parameter vector $(C_0^\ell,\ldots,C_6^{\ell})\in\R^7$ depends on the original coefficients of perturbation in a surjective way. From Propositions \ref{pro2} and \ref{pro3}, $\operatorname{Span}(\mathcal{F}_{6}^{k})$ is an ET-system with accuracy 1 on $[a,b]$ for any $0<a<b.$ Thus, we conclude that the maximum number of positive simple zeros of $Q_\ell^k(x)$ is 7 and there exists $(C_0^\ell,\ldots,C_6^{\ell})\in\R^7$ for which $Q_\ell^k(x)$ has exactly 7 positive simple zeros. Therefore,  $m_{\ell}(2k)=7$ for $k>1$ and $\ell=2,\dots 6$.

\subsection*{Case 3: } Let $n=2k+1$, $k>0$ and $\ell\{2,\dots, 5\}.$  Assuming conditions on the parameters of perturbations such that $M_i=0$ for $i\in\{1,\ldots,\ell-1\},$ the Melnikov function of order $\ell$ is provided by
\begin{equation*}\label{mi2k+1}
M_\ell(x)=\dfrac{1}{(1+(1+2k)x^{4k})^2}R_\ell^k(x),
\end{equation*}
where
\[\begin{array}{rcl}
R_\ell^k(x)&=&
C_0^\ell u_1^k(x)+C_1^\ell u_4^k(x)+C_2^\ell u_7^k(x)+C_3^\ell u_{8}^k(x)+C_4^\ell u_{10}^k(x)+C_5^\ell u_{5}^k(x)+C_6^\ell u_{11}^k(x)\\
& &+C_7^\ell u_{14}^k(x),
\end{array}
\]
which belongs to $\operatorname{Span}(\mathcal{F}_{5}^{k})$ (see Table \ref{table3}). In addition, it can be seen that the parameter vector $(C_0^\ell,\ldots,C_7^{\ell})\in\R^8$ depends on the original coefficients of perturbation in a surjective way. From Propositions \ref{pro1}, $\operatorname{Span}(\mathcal{F}_{5}^{k})$ is an ECT-system on $[a,b]$ for any $0<a<b.$ Thus, we conclude that the maximum number of positive simple zeros of $R_\ell^k(x)$ is 7 and there exists $(C_0^\ell,\ldots,C_7^{\ell})\in\R^8$ for which $R_\ell^k(x)$ has exactly 7 zeros. Therefore, $m_\ell(2k+1)=7$ for $k>0$ and $\ell=1,\dots,5$.

\subsection*{Case 4: }  Let $n=2k+1$, $k>0,$ and $\ell=6.$ Assuming conditions on the parameters of perturbations such that $M_i=0$ for $i\in\{1,\ldots,5\},$ the Melnikov function of order $6$ has two possible forms (see Table \ref{table3}). The first one has its numerator as a linear combination of functions in $\mathcal{F}_{5}^{k},$ which, from Propositions \ref{pro1}, has at most $7$ positive simple zeros.
The second one is provided by
\[\label{2m62k+1}
M_6(x)=\dfrac{L^k(x^{2k})}{x^2 (1+(1+2k)x^{4k})^2},
\]
where
$$L^k(x)=C_0u_{18}^k(x)+C_1u_{19}^k(x)+C_2u_{20}^k(x)+C_3u_{21}^k(x)+C_4u_{22}^k(x)+C_5u_{23}^k(x)+C_6u_{24}^{k,\lambda}(x),$$
which belongs to  $\operatorname{Span}(\mathcal{F}_{7}^{k,\lambda})$ (see Table \ref{table3}).  In addition, it can be seen that the parameter vector $(C_0^\ell,\ldots,C_6^{\ell})\in\R^7$ depends on the original coefficients of perturbation in a surjective way. For $k=1$, Proposition \ref{pro4} provides that $L^1(x)$ has at most $10$ positive simple zeros and that there exists $(C_0^\ell,\ldots,C_6^{\ell})\in\R^7$ such that $L^1(x)$ has at least $8$ positive simple zeros. For $k>1,$ Proposition \ref{pro5} provides that  $L^k(x)$ has at most $14$ positive simple zeros and that there exists $(C_0^\ell,\ldots,C_6^{\ell})\in\R^7$ such that $L^K(x)$ has at least $9$ positive simple zeros. Therefore, $8\leq m_{6}(3)\leq 10$ and, for $k>1,$ $9\leq m_{6}(2k+1)\leq 14.$

\bigskip

Hence, we have concluded the proof of Theorem \ref{Theorem-Melnikov2}.

\section*{Appendix: Proof of Theorem \ref{thm:melnikov}}

Let $\varphi(t,x,\varepsilon)$ denote the solution of the $T$-periodic nonsmooth differential system \eqref{general-system2} with initial condition $\varphi(0,x,\varepsilon)=x.$ 
Let $\alpha_j(x,\varepsilon)$ denote the smallest positive time for which the trajectory $\varphi_{j-1}(\cdot,x,\varepsilon)$, starting at $\varphi_{j-1}(\alpha_{j-1}(x,\varepsilon),x,\varepsilon)\in D$, reaches the manifold $\lbrace (\theta_j(x),x) : x\in D\rbrace\subset \Sigma$. In this way
\begin{equation}\label{defalphaj}
\alpha_j(x,\varepsilon)=\theta_j(\varphi_{j-1}(\alpha_j(x,\varepsilon),x,\varepsilon)),
\end{equation}
for $j=1,\dots,N$. For the sake of completeness, denote $\alpha_0(x,\varepsilon)=0$.
Thus, 
\begin{equation}\label{solEdo}
\varphi(t,x,\varepsilon)=\left\lbrace\begin{array}{cc}
\varphi_0(t,x,\varepsilon),& 0\leq t\leq\alpha_1(x,\varepsilon),\\
\varphi_1(t,x,\varepsilon),&\alpha_1(x,\varepsilon)\leq t\leq\alpha_2(x,\varepsilon),\\
\vdots&\vdots\\
\varphi_N(t,x,\varepsilon),&\alpha_N(x,\varepsilon)\leq t\leq T,
\end{array}\right.
\end{equation}
where
\[\label{edo2}
\left\lbrace\begin{array}{ll}
\dfrac{\partial \varphi_j}{\partial t}(t,x,\varepsilon)=F^j(t,\varphi_j(t,x,\varepsilon),\varepsilon),& \text{for } j=0,\dots,N,\\
\varphi_0(0,x,\varepsilon)=x,\\
\varphi_j(\alpha_j(x,\varepsilon),x,\varepsilon)=\varphi_{j-1}(\alpha_j(x,\varepsilon),x,\varepsilon), &\text{for }  j=1,\dots,N.
\end{array}\right.
\]

The recurrence above describes initial value problems, which are equivalent to the following integral equations:
\begin{equation}\label{pviRec}
\left\lbrace\begin{array}{l}
\varphi_0(t,x,\varepsilon)=x+\displaystyle\int_0^tF^0(s,\varphi_0(s,x,\varepsilon),\varepsilon)ds,\\
\varphi_j(t,x,\varepsilon)=\varphi_{j-1}(\alpha_{j}(x,\varepsilon),x,\varepsilon)+\displaystyle\int_{\alpha_{j}(x,\varepsilon)}^t F^j(x,\varphi_j(s,x,\varepsilon),\varepsilon)ds, \text{ for } j=1,\dots,N.
\end{array}\right.
\end{equation}

Now, consider the displacement function
\begin{equation}\label{Delta}
	\Delta(x,\varepsilon)=\varphi(T,x,\varepsilon)-x.
\end{equation}
By denoting $z_i^j(t,x)=\dfrac{\p^i \varphi_j}{\p \e^i}(t,x,0),$ we expand $\varphi_j(t, x, \varepsilon)$, around $\varepsilon= 0$ up to power $k$, we get that
\begin{equation}\label{varphiexp}
\varphi_j(t,x,\varepsilon)=x+\displaystyle\sum_{i=1}^k\dfrac{\varepsilon^i}{i!}z_i^j(t,x)+O(\varepsilon^{k+1}),
\end{equation}
and, consequently,
\begin{equation*}\label{Delta2}
	\begin{array}{rl}
	\Delta(x,\varepsilon)=\displaystyle\sum_{i=1}^k\dfrac{\varepsilon^i}{i!}z_i^N(T,x)+O(\varepsilon^{k+1}).
	\end{array}
\end{equation*}

Hence, the Melnikov function of order $i$ is provided by $$M_i(x)=\dfrac{1}{i!}z_i^N(T,x).$$ Indeed, from \eqref{Delta}, it is clear that  $T$-periodic solutions $\f(t,x,\varepsilon)$ of system \eqref{general-system2}, satisfying $x(0,x,\varepsilon) = x$, are in one-to-one correspondence to the zeros of the equation $\Delta(x,\varepsilon) = 0$. 
	From hypothesis, 
	\[
	\widehat{\Delta}(x,\varepsilon):=\dfrac{\Delta(x,\varepsilon)}{\varepsilon^\ell} = M_\ell(x)+\CO(\e^{\ell+1}),
	\]
	$\widehat{\Delta}(a^*,0) = M_l(a^*) = 0$, and $\det\left(\dfrac{\partial\widehat{\Delta}}{\partial x}(a^*,0)\right) = \det(DM_\ell(a^*)) \neq 0$.
	Therefore, based on the Implicit Function Theorem, we get the existence of a unique $C^k$ function $a(\varepsilon)\in D,$ defined for $|\varepsilon| \neq 0$ sufficiently small, such that $a(0) = a^*$ and $\Delta(a(\varepsilon),\varepsilon) = \widehat{\Delta}(a(\varepsilon), \varepsilon) = 0$.

We conclude the proof of Theorem  \ref{thm:melnikov} by showing in Proposition \ref{teoremaztil} that the functions $z_i^j(t,x)$ are provided by \eqref{ztilde}, \eqref{alphaj}, and \eqref{omegaj}. For this, we need the following technical lemma: 
\begin{lemma}\label{Ql}
	Let $Q_l:\mathbb{R}^d\times\dots\times\mathbb{R}^d\rightarrow\mathbb{R}^d$ be a $l$-multilinear map. Then,
	\begin{equation}\label{Q}
		Q_l\left(\displaystyle\sum_{i=1}^k\varepsilon^i x_i\right)^l=\displaystyle\sum_{p=l}^{kl}\varepsilon^p\displaystyle\sum_{u\in S_{p,l}}Q_l\left(\prod_{r=1}^l x_{u_r}\right),
	\end{equation}
	where $S_{p,l}=\left\lbrace(u_1,\dots,u_l)\in (\mathbb{Z}^+)^l: u_1+\dots+u_l=p\right\rbrace$. 
\end{lemma}
\begin{proof}
	The proof of this result will follow by induction on $l$. It's easy to see that for $l=1$ the result holds. Suppose by induction hypothesis that \eqref{Q} holds for $l$-multlinear maps. Then, define the $l$-multlinear map 
$$\widetilde{Q}_{l}(y_1,\dots,y_{l})=Q_{l+1}\left(y_1,\dots,y_{l},\displaystyle\sum_{i=1}^k\varepsilon^ix_i\right).$$
	Notice that
	$$Q_{l+1}\left(\displaystyle\sum_{i=1}^k\varepsilon^i x_i\right)^{l+1} =\widetilde{Q}_{l}\left(\displaystyle\sum_{i=1}^k\varepsilon^i x_i\right)^{l}.$$
	Thus, applying the induction hypothesis for $\widetilde{Q}_l$ in the equality above, we have
	\[
	\begin{array}{rl}
	Q_{l+1}\left(\displaystyle\sum_{i=1}^k\varepsilon^i x_i\right)^{l+1} 
	&=\displaystyle\sum_{p=l}^{kl}\varepsilon^p\displaystyle\sum_{u\in S_{p,l}} Q_{l+1}\left(\displaystyle\prod_{r=1}^{l} x_{u_r},\displaystyle\sum_{i=1}^k\varepsilon^i x_i\right)\\
		&=\displaystyle\sum_{i=1}^k\displaystyle\sum_{p=l}^{kl}\varepsilon^{p+i}\displaystyle\sum_{u\in S_{p,l}}Q_{l+1}\left(\displaystyle\prod_{r=1}^{l} x_{u_r}, x_i\right)\\
	&=\displaystyle\sum_{i=1}^k\displaystyle\sum_{q=l+i}^{kl+i}\varepsilon^{q}\displaystyle\sum_{u\in S_{q-i,l}}Q_{l+1}\left(\displaystyle\prod_{r=1}^{l} x_{u_r}, x_i\right)\\
	&=\displaystyle\sum_{q=l+1}^{k(l+1)}\varepsilon^{q}\displaystyle\sum_{i=1}^{q-l}\displaystyle\sum_{u\in S_{q-i,l}}Q_{l+1}\left(\displaystyle\prod_{r=1}^{l} x_{u_r}, x_i\right).\\
	\end{array}
\]
Considering $S_{i,q,l+1}=\lbrace(v_1,\dots,v_l,i);(v_1,\dots,v_l)\in S_{q-i,l}\rbrace $, we get
		\begin{equation*}\label{qaux}
		Q_{l+1}\left(\displaystyle\sum_{i=1}^k\varepsilon^i x_i\right)^{l+1} =\displaystyle\sum_{q=l+1}^{k(l+1)}\varepsilon^{q}\displaystyle\sum_{i=1}^{q-l}\displaystyle\sum_{v\in S_{i,q,l+1}}Q_{l+1}\left(\displaystyle\prod_{r=1}^{l+1} x_{v_r}\right).
	\end{equation*}
Thus, since
	\begin{equation*}\label{sss}
	S_{q,l+1}=\overset{q-l}{\underset{i=1}{\dot{\bigcup}}}S_{i,q,l+1},
	\end{equation*}
we conclude that
	$$
	Q_{l+1}\left(\displaystyle\sum_{i=1}^k\varepsilon^i x_i\right)^{l+1} =\displaystyle\sum_{q=l+1}^{k(l+1)}\varepsilon^{q}\displaystyle\sum_{{v}\in S_{q,l+1}}Q\left(\displaystyle\prod_{r=1}^{l+1} x_{v_r}\right),
	$$
	which finishes this proof.
\end{proof}

\begin{proposition}\label{teoremaztil}
	The functions $z_i^j(t,x)$ are provided by \eqref{ztilde}, \eqref{alphaj}, and \eqref{omegaj}.
\end{proposition}
\begin{proof}
First of all, recall the Faà di Bruno's formula for the $l$th derivative of the composed function: Let g and h be sufficiently smooth functions then
\begin{equation}\label{Faa}
	\dfrac{d^l}{d\alpha^l}g(h(\alpha))=\displaystyle\sum_{b\in S_l}\dfrac{l!}{b_1!b_2!2!^{b_2}\dots b_l!l!^{b_l}}g^{(L_b)}(h(\alpha))\displaystyle\prod_{j=1}^l\left( h^{(j)}(\alpha)\right)^{b_j},
\end{equation}  where $S_l$ is the set of all l-tuples of non-negative integers
$(b_1, b_2, \dots, b_l)$ satisfying $b_1 + 2b_2 + \dots + lb_l = l$, and $L =
b_1 + b_2 + \dots + b_l$.

	Here, we shall we expand $\varphi_j(t,x,\varepsilon)$, around $\varepsilon=0$ up to order $k$. By taking \eqref{Fj} into account and computing the expansion of $F_i^j(s,\varphi_j(s,x,\varepsilon))$ around $\varepsilon=0$ up to order $k-i$, we obtain
	\begin{equation}\label{intFij3}
	\begin{array}{rl}
	\displaystyle\int_{\alpha_{j}(x,\varepsilon)}^t F^j(s,\varphi_j(s,x,\varepsilon),\varepsilon)ds=&\displaystyle\int_{\alpha_{j}(x,\varepsilon)}^t \left(\displaystyle\sum_{i=1}^k \displaystyle\sum_{l=0}^{k-i}\dfrac{\varepsilon^{i+l}}{l!} \dfrac{\partial^l}{\partial\varepsilon^l}\left(F_i^j(s,\varphi_j(s,x,\varepsilon))\right)\Big|_{\varepsilon=0}\right) ds+O(\varepsilon^{k+1})\\
	=&\displaystyle\sum_{i=1}^k\varepsilon^{i} \displaystyle\int_{\alpha_{j}(x,\varepsilon)}^t \left(\displaystyle\sum_{l=0}^{i-1}\dfrac{1}{l!} \dfrac{\partial^l}{\partial\varepsilon^l}\left(F_{i-l}^j(s,\varphi_j(s,x,\varepsilon))\right)\Big|_{\varepsilon=0}\right) ds+O(\varepsilon^{k+1}).
	\end{array}
	\end{equation}
	For $i=1,\dots,k$, and $j=0,\dots,N$, denote
	\[
	K_i^j(t,x)=\displaystyle\sum_{l=0}^{i-1}\dfrac{1}{l!}\dfrac{\partial^l}{\partial\varepsilon^l}\left(F_{i-l}^j(t,\varphi_j(t,x,\varepsilon))\right)\Big|_{\varepsilon=0}.\]
By applying Faà di Bruno's formula \eqref{Faa} in the expression above, it follows that
	\begin{equation}\label{Kijexp}
		\begin{array}{rcl}
			K_1^j(t,x)&=&F_1^j(t,x),\\
			K_i^j(t,x)&=&F_{i}^j(t,x)+\displaystyle\sum_{l=1}^{i-1}\displaystyle\sum_{b\in S_l}\dfrac{1}{b_1!b_2!2!^{b_2}\dots b_l!l!^{b_l}}\partial^{L_b}_xF_{i-l}^j(t,x)\displaystyle\prod_{m=1}^l\left(z_m^j(t,x) \right)^{b_m},
		\end{array}
	\end{equation}
	for $i=2,\dots,k$, and $j=0,\dots,N$, where $L_b$ and $S_l$ are defined in \eqref{Faa}.
	
	Now, expanding $\displaystyle\int_{\alpha_{j}(x,\varepsilon)}^t K_{i}^j(s,x) ds$ around $\varepsilon=0$ up to order $k-i$, we get
	\begin{equation}\label{phij3}
		\begin{array}{rcl}
			\displaystyle\sum_{i=1}^k \varepsilon^i\displaystyle\int_{\alpha_{j}(x,\varepsilon)}^t K_{i}^j(s,x) ds&=&\displaystyle\sum_{i=1}^{k} \varepsilon^{i}\left( \displaystyle\sum_{p=0}^{k-i}\dfrac{\varepsilon^p}{p!}\dfrac{\partial^p}{\partial\varepsilon^p}\left( \displaystyle\int_{\alpha_{j}(x,\varepsilon)}^t K_i^j(s,x) ds\right)\Bigg|_{\varepsilon=0}+O(\varepsilon^{k-i+1})\right) \\
			&=&\displaystyle\sum_{i=1}^{k}  \displaystyle\sum_{p=0}^{k-i}\dfrac{\varepsilon^{i+p}}{p!}\dfrac{\partial^p}{\partial\varepsilon^p}\left( \displaystyle\int_{\alpha_{j}(x,\varepsilon)}^t K_i^j(s,x) ds\right)\Bigg|_{\varepsilon=0}+O(\varepsilon^{k+1}) \\
			&=&\displaystyle\sum_{i=1}^{k} \varepsilon^{i} \displaystyle\sum_{p=0}^{i-1}\dfrac{1}{p!}\dfrac{\partial^p}{\partial\varepsilon^p}\left( \displaystyle\int_{\alpha_{j}(x,\varepsilon)}^t K_{i-p}^j(s,x) ds\right)\Bigg|_{\varepsilon=0}+O(\varepsilon^{k+1}).
		\end{array}
	\end{equation}
	For $i=1,\dots, k$, and $j=0,\dots, N$, denote
	\begin{equation}\label{Iij}
		\begin{array}{l}
			I_i^j(t,x)=\displaystyle\sum_{p=0}^{i-1}\dfrac{1}{p!}\dfrac{\partial^p}{\partial\varepsilon^p}\left( \displaystyle\int_{\alpha_{j}(x,\varepsilon)}^t K_{i-p}^j(s,x) ds\right)\Bigg|_{\varepsilon=0}.
		\end{array}
	\end{equation}
	Thus,
	\begin{equation}\label{Iijexp}
		\begin{array}{rcll}
			I_1^j(t,x)&=& \displaystyle\int_{\theta_{j}(x)}^t K_{1}^j(s,x) ds,&\text{for } j=0,\ldots,N,\\
			I_i^j(t,x)&=& \displaystyle\int_{\theta_{j}(x)}^t K_{i}^j(s,x) ds+\widetilde{K}_i^j(x),&\text{for } i=2,\dots,k,\text{ and } j=0,\ldots,N,
		\end{array}
	\end{equation}
	provided that
	\begin{equation}\label{Ktilij}
		\widetilde{K}_i^j(x)=- \displaystyle\sum_{p=1}^{i-1}\dfrac{1}{p!}\dfrac{\partial^{p-1}}{\partial\varepsilon^{p-1}}\left( K_{i-p}^j(\alpha_{j}(x,\varepsilon),x) \dfrac{\partial}{\partial\varepsilon}\alpha_j(x,\varepsilon)\right)\Bigg|_{\varepsilon=0},
	\end{equation}
	for $i=1,\dots,k,\text{ and } j=0,\ldots,N.$
	
	Replacing \eqref{Iij} into \eqref{phij3} and, then, into \eqref{intFij3}, we get 
	\[\displaystyle\int_{\alpha_{j}(x,\varepsilon)}^t F^j(s,\varphi_j(s,x,\varepsilon),\varepsilon)ds=\displaystyle\sum_{i=1}^k\varepsilon^{i} I_i^j(t,x)+O(\varepsilon^{k+1}).\]
	Thus, replacing the expression above into \eqref{pviRec}, we obtain
	\[
	\left\lbrace\begin{array}{l}
	\varphi_0(t,x,\varepsilon)=x+\displaystyle\sum_{i=1}^k\varepsilon^i I_i^0(t,x)+O(\varepsilon^{k+1}),\\
	\varphi_j(t,x,\varepsilon)=\varphi_{j-1}(\alpha_{j}(x,\varepsilon),x,\varepsilon)+\displaystyle\sum_{i=1}^k\varepsilon^i I_i^j(t,x)+O(\varepsilon^{k+1}),
	\end{array}\right.\]
	for $j=1,\dots,N$. Hence, proceeding by induction on $j$, we conclude that
	\begin{equation}\label{claim}
	\begin{array}{rl}
	\varphi_j(t,x,\varepsilon)=x+\displaystyle\sum_{i=1}^k\varepsilon^iJ_i^j(t,x,\varepsilon)+O(\varepsilon^{k+1}),&\text{for } j=0,\dots,N,
	\end{array}
	\end{equation}
	where
	\begin{equation}\label{wij}
	\left\lbrace\begin{array}{ll}
	J_i^0(t,x,\varepsilon)=I_i^0(t,x),\\
	J_i^j(t,x,\varepsilon)=\displaystyle\sum_{l=0}^{j-1}I_i^{l}(\alpha_{l+1}(x,\varepsilon),x)+I_i^j(t,x),&
	\end{array}\right.
	\end{equation}
	for $i=1,\dots,k, \text{and }j=1,\dots,N.$
	
	Now, expanding $J_i^{j}(t,x,\varepsilon)$ around $\varepsilon=0$ up to order $k-i$, we get
	\begin{equation}\label{wij4}
		\begin{array}{rcl}
			\displaystyle\sum_{i=1}^k\varepsilon^iJ_i^{j}(t,x,\varepsilon)&=&\displaystyle\sum_{i=1}^k\varepsilon^i\left(\displaystyle\sum_{p=0}^{k-i}\dfrac{\varepsilon^p}{p!}\dfrac{\partial^p}{\partial\varepsilon^p}\left( J_i^{j}(t,x,\varepsilon)\right)\Big|_{\varepsilon=0}+O(\varepsilon^{k-i-1})\right)\\
			&=&\displaystyle\sum_{i=1}^k\varepsilon^i\displaystyle\sum_{p=0}^{i-1}\dfrac{1}{p!}\dfrac{\partial^p}{\partial\varepsilon^p}\left( J_{i-p}^{j}(t,x,\varepsilon)\right)\Big|_{\varepsilon=0}+O(\varepsilon^{k+1}).
	\end{array}
	\end{equation}
	Therefore, replacing \eqref{wij4} into \eqref{claim} and taking \eqref{varphiexp} into account, it follows that
	\begin{equation}\label{ztilijs}
	z_i^j(t,x)=i!\,\displaystyle\sum_{p=0}^{i-1}\dfrac{1}{p!}\dfrac{\partial^p}{\partial\varepsilon^p}\left( J_{i-p}^{j}(t,x,\varepsilon)\right)\Big|_{\varepsilon=0}.
	\end{equation}
	In particular, replacing \eqref{Kijexp} into \eqref{Iijexp}, we get 
	\begin{equation}\label{ztil1j}
		z_{1}^{j}(t,x)=\displaystyle\int_0^tF_1(s,x)ds.
	\end{equation}		
	Now, for $i=2,\dots,k$, and $j=0$, replacing \eqref{Iijexp} into \eqref{wij}, and then into \eqref{ztilijs}, we obtain
	\begin{equation}\label{ztil0i}
		z_i^0(t,x)=i!\displaystyle\sum_{l=0}^{j-1}\displaystyle\int_{\theta_l(x)}^{\theta_{l+1}(x)}K_i^l(s,x)ds+i!\displaystyle\int_{\theta_j(x)}^tK_i^j(s,x)ds.
	\end{equation}	
	Finally, for $j=1,\dots,N$, replacing \eqref{wij} into \eqref{ztilijs}, it writes
	\begin{equation}\label{zijcome}
		\begin{array}{rcl}
			z_i^j(t,x)&=&i!\displaystyle\sum_{p=0}^{i-1}\dfrac{1}{p!}\dfrac{\partial^p}{\partial\varepsilon^p}\left( \displaystyle\sum_{l=0}^{j-1}I_{i-p}^l(\alpha_{l+1}(x,\varepsilon),x)+I_{i-p}^j(t,x)\right)\Bigg|_{\varepsilon=0}\\
			&=&i!\left( \displaystyle\sum_{a=0}^{j-1} I_{i}^{a}(\theta_{a+1}(x),x)+I_{i}^j(t,x)\right) \\
			& &+i!\displaystyle\sum_{p=1}^{i-1}\dfrac{1}{p!}\displaystyle\sum_{a=0}^{j-1}\dfrac{\partial^p}{\partial\varepsilon^p}\left( I_{i-p}^{a}(\alpha_{a+1}(x,\varepsilon),x)\right)\Big|_{\varepsilon=0}.
		\end{array}
	\end{equation}
	From \eqref{Iijexp}, we have
	\begin{equation}\label{soma}
		\displaystyle\sum_{a=0}^{j-1} I_{i}^{a}(\theta_{a+1}(x),x)+I_{i}^j(t,x)=\displaystyle\sum_{a=0}^{j-1} \displaystyle\int_{\theta_{a}(x)}^{\theta_{a+1}(x)} K_i^a(s,x) ds+\displaystyle\int_{\theta_{j}(x)}^t K_i^j(s,x) ds+\displaystyle\sum_{a=0}^{j}\widetilde{K}_i^a(x) 
	\end{equation}
	and
	\begin{equation}\label{zijcomes}
		\begin{array}{l}
			\displaystyle\sum_{p=1}^{i-1}\dfrac{1}{p!}\dfrac{\partial^p}{\partial\varepsilon^p}\left( I_{i-p}^{a}(\alpha_{a+1}(x,\varepsilon),x)\right)\Big|_{\varepsilon=0}=\displaystyle\sum_{p=1}^{i-1}\dfrac{1}{p!}\dfrac{\partial^{p-1}}{\partial\varepsilon^{p-1}}\left( K_{i-p}^a(\alpha_{a+1}(x,\varepsilon),x)\dfrac{\partial}{\partial\varepsilon}\alpha_{a+1}(x,\varepsilon)\right)\Big|_{\varepsilon=0}.
		\end{array}
	\end{equation}
	From \eqref{solEdo}, we have $\alpha_0(x,\varepsilon)=0$, then from this and \eqref{Ktilij}, we obtain
	\begin{equation}\label{wtijaux}
		\begin{array}{l}
			\displaystyle\sum_{a=0}^{j}\widetilde{K}_{i}^{a}(x)+\displaystyle\sum_{a=0}^{j-1}\displaystyle\sum_{p=1}^{i-1}\dfrac{1}{p!}\dfrac{\partial^{p-1}}{\partial\varepsilon^{p-1}}\left(K_{i-p}^a(\alpha_{a+1}(x,\varepsilon),x)\dfrac{\partial}{\partial\varepsilon}\alpha_{a+1}(x,\varepsilon)\right)\Big|_{\varepsilon=0}\\
			=\displaystyle\sum_{a=1}^{j}\displaystyle\sum_{p=1}^{i-1}\dfrac{1}{p!}\dfrac{\partial^{p-1}}{\partial\varepsilon^{p-1}}\left(K_{i-p}^{a-1}(\alpha_{a}(x,\varepsilon),x)\dfrac{\partial}{\partial\varepsilon}\alpha_{a}(x,\varepsilon)\right)\Big|_{\varepsilon=0}\\
			-\displaystyle\sum_{a=0}^{j}\displaystyle\sum_{p=1}^{i-1}\dfrac{1}{p!}\dfrac{\partial^{p-1}}{\partial\varepsilon^{p-1}}\left(K_{i-p}^a(\alpha_a(x,\varepsilon),x)\dfrac{\partial\alpha_a}{\partial\varepsilon}(x,\varepsilon)\right)\Big|_{\varepsilon=0}\\
			=\displaystyle\sum_{a=1}^{j}\displaystyle\sum_{p=1}^{i-1}\dfrac{1}{p!}\dfrac{\partial^{p-1}}{\partial\varepsilon^{p-1}}\left(\left( K_{i-p}^{a-1}(\alpha_{a}(x,\varepsilon),x)-K_{i-p}^a(\alpha_a(x,\varepsilon),x)\right)\dfrac{\partial\alpha_{a}}{\partial\varepsilon}(x,\varepsilon)\right)\Big|_{\varepsilon=0}.
\end{array}
\end{equation}
	
	Therefore, from \eqref{ztil1j}, \eqref{ztil0i}, \eqref{zijcome}, \eqref{soma}, \eqref{zijcomes}, and \eqref{wtijaux}, we have that
	
	\begin{equation}\label{sist}
	\begin{array}{rcl}
	z_1^j(t,x)&=&\displaystyle\int_0^tF_1(s,x)ds,\\
	z_i^0(t,x)&=&i!\displaystyle\int_{\theta_j(x)}^tK_i^j(s,x)ds,\\
	z_i^j(t,x)&=&i!\displaystyle\sum_{l=0}^{j-1}\displaystyle\int_{\theta_l(x)}^{\theta_{l+1}(x)}K_i^l(s,x)ds+i!\displaystyle\int_{\theta_j(x)}^tK_i^j(s,x)ds\\
	& &	+i!\displaystyle\sum_{a=1}^{j}\displaystyle\sum_{p=1}^{i-1}\dfrac{1}{p!}\dfrac{\partial^{p-1}}{\partial\varepsilon^{p-1}}\left(\left( K_{i-p}^{a-1}(\alpha_{a}(x,\varepsilon),x)-K_{i-p}^a(\alpha_a(x,\varepsilon),x)\right)\dfrac{\partial\alpha_{a}}{\partial\varepsilon}(x,\varepsilon)\right)\Big|_{\varepsilon=0},
	\end{array}
	\end{equation}
	for $i=2,\dots,k$ and $j=1,\dots,N$.
	Notice that 
	\begin{equation*}\label{Kijf}
		K_i^j(t,x)=\dfrac{1}{i!}\dfrac{\partial z_i^j}{\partial t}(t,x), \text{ for } i=1,\dots,k,\text{ and } j=0,\dots,N.
	\end{equation*}
	Then, denoting $\delta_i^j(t,x)=\dfrac{1}{i!}\left( z_i^{j-1}(t,x)-z_i^j(t,x)\right) $, we get that
	\begin{equation}\label{delta}
		\begin{array}{rl}
			\dfrac{\partial^p}{\partial\varepsilon^p}\left( \delta_{i-p}^a(A_a^p(x,\varepsilon),x)\right) \Big|_{\varepsilon=0}=&\dfrac{\partial^p}{\partial\varepsilon^p}\left( \delta_{i-p}^a(\alpha_a(x,\varepsilon),x)\right) \Big|_{\varepsilon=0}\\
			=&\dfrac{\partial^{p-1}}{\partial\varepsilon^{p-1}}\left(\left( K_{i-p}^{a-1}(\alpha_{a}(x,\varepsilon),x)-K_{i-p}^a(\alpha_a(x,\varepsilon),x)\right)\dfrac{\partial\alpha_{a}}{\partial\varepsilon}(x,\varepsilon)\right)\Big|_{\varepsilon=0}
		\end{array}
	\end{equation}
	where $A_a^p(x,\varepsilon)=\displaystyle\sum_{q=0}^p\dfrac{\varepsilon^q}{q!}\alpha_a^q(x)$.
	Moreover,
	\begin{equation}\label{faadi}
		\begin{array}{l}
			\displaystyle\sum_{a=0}^{j-1} \displaystyle\int_{\theta_{a}(x)}^{\theta_{a+1}(x)} K_i^a(s,x) ds+\displaystyle\int_{\theta_{j}(x)}^t K_i^j(s,x) ds=\\
			=\displaystyle\int_0^t \left( F_{i}(s,x)+\displaystyle\sum_{l=1}^{i-1}\displaystyle\sum_{b\in S_l}B_b\partial^{L_b}_xF_{i-l}(s,x)\displaystyle\prod_{m=1}^l\left(z_m(s,x) \right)^{b_m}\right)  ds
		\end{array}
	\end{equation}
	where $$
	z_i(t,x)=\left\{\begin{array}{cc}
	z_i^0(t,x),& 0\leq t\leq \theta_1(x),\\
	z_i^2(t,x),&\theta_1(x)\leq t\leq\theta_2(x),\\
	\vdots&\vdots\\
	z_i^N(t,x),&\theta_N(x)\leq t\leq T.
	\end{array}\right.$$
	
	Hence, from \eqref{sist}, \eqref{delta}, and \eqref{faadi}, it writes
	\begin{equation}\label{final}
	\begin{array}{rcl}
	z_1^j(t,x)&=&\displaystyle\int_0^tF_1(s,x)ds,\\
	z_i^0(t,x)&=&i!\displaystyle\int_0^t \left( F_{i}^0(s,x)+\displaystyle\sum_{l=1}^{i-1}\displaystyle\sum_{b\in S_l}B_b\partial^{L_b}_xF_{i-l}^0(s,x)\displaystyle\prod_{m=1}^l\left(z_m^0(s,x) \right)^{b_m}\right)  ds,\\
	z_i^j(t,x)&=&i!\displaystyle\int_0^t \left( F_{i}(s,x)+\displaystyle\sum_{l=1}^{i-1}\displaystyle\sum_{b\in S_l}B_b\partial^{L_b}_xF_{i-l}(s,x)\displaystyle\prod_{m=1}^l\left(z_m(s,x) \right)^{b_m}\right)  ds\\
	& &+i!\displaystyle\sum_{a=1}^{j}\displaystyle\sum_{p=1}^{i-1}\dfrac{1}{p!}\dfrac{\partial^{p}}{\partial\varepsilon^{p}}\left(\delta_{i-p}^a(A_a^p(x,\varepsilon),x)\right)\Big|_{\varepsilon=0},
	\end{array}
	\end{equation}
	for $i=2,\dots,k$ and $j=1,\dots,N$. Notice that the formula \eqref{ztilde} follows from \eqref{final} by induction on $j$. Therefore, the proof of Proposition \ref{teoremaztil} is concluded by proving the following claim:

\begin{claim}\label{claim1}
For $q=1,\dots,k$ and $j=1,\dots,N$, we have that $\dfrac{\partial^q \alpha_j}{\partial\varepsilon^q}(x,0)=\alpha_j^q(x)$ where $\alpha_j^q$ is provided by \eqref{alphaj}.
\end{claim}
Indeed, from \eqref{defalphaj}, we get
\[
\alpha_j^q(x)=\dfrac{\partial^q}{\partial\varepsilon^q}\left(\theta_j(\varphi_{j-1}(\alpha_j(x,\varepsilon),x,\varepsilon))\right)\Big|_{\varepsilon=0}.
\]
From \eqref{varphiexp} and the above expression, we obtain
$$\alpha_j^q(x)=\dfrac{\partial^q}{\partial\varepsilon^q}\left(\theta_j\left(x+h(x,\varepsilon)\right)\right)\Big|_{\varepsilon=0},$$
where
\begin{equation*}\label{h}
h(x,\varepsilon)=\displaystyle\sum_{i=1}^k\dfrac{\varepsilon^i}{i!}z_i^{j-1}(\alpha_j(x,\varepsilon),x)+O\left(\varepsilon^{k+1}\right).
\end{equation*}

Computing the expansion of $h(x,\varepsilon)$, around $\varepsilon=0$ up to order $k-i$, we get
\begin{equation}\label{h2}
	\begin{array}{rcl}
		h(x,\varepsilon)&=&\displaystyle\sum_{i=1}^k\dfrac{\varepsilon^i}{i!}\displaystyle\sum_{a=0}^{k-i}\dfrac{\varepsilon^a}{a!}\dfrac{\partial^a}{\partial\varepsilon^a}\left(z_i^{j-1}(\alpha_j(x,\varepsilon),x)\right)\Big|_{\varepsilon=0}+O(\varepsilon^{k+1})\\
		&=&\displaystyle\sum_{i=1}^k\varepsilon^i\displaystyle\sum_{a=0}^{i-1}\dfrac{1}{(i-a)!a!}\dfrac{\partial^a}{\partial\varepsilon^a}\left(z_{i-a}^{j-1}(\alpha_j(x,\varepsilon),x)\right)\Big|_{\varepsilon=0}+O(\varepsilon^{k+1}).\\
\end{array}
\end{equation}

For $i=1,\dots,k$, and $j=1,\dots,N$, denote 
\begin{equation}\label{wchapij}
w_i^j(x)=\displaystyle\sum_{a=0}^{i-1}\dfrac{1}{(i-a)!a!}\dfrac{\partial^a}{\partial\varepsilon^a}\left(z_{i-a}^{j-1}(\alpha_j(x,\varepsilon),x)\right)\Big|_{\varepsilon=0}.
\end{equation}

Expanding $\theta_j(x+h(x,\varepsilon))$ in Taylor series in $h(x,\varepsilon)$, around $h(x,\varepsilon)=0$ up to order $k$, we have
\begin{equation}\label{derialps2}
	\begin{array}{rcl}
	\alpha_j^q(x)&=&\dfrac{\partial^q}{\partial\varepsilon^q}\left(\theta_j\left(x\right)+\displaystyle\sum_{l=1}^k\dfrac{1}{l!}D^l\theta_j(x)(h(x,\varepsilon))^l+O\left((h(x,\varepsilon))^{k+1}\right)\right)\Bigg|_{\varepsilon=0}\\
	&=&\dfrac{\partial^q}{\partial\varepsilon^q}\left(\displaystyle\sum_{l=1}^k\dfrac{1}{l!}D^l\theta_j(x)(h(x,\varepsilon))^l\right)\Bigg|_{\varepsilon=0}.
	\end{array}
\end{equation}
Thus, replacing \eqref{wchapij} into \eqref{h2}, and into \eqref{derialps2}, we obtain
\[
\begin{array}{l}
\alpha_j^q(x)=\dfrac{\partial^q}{\partial\varepsilon^q}\left(\displaystyle\sum_{l=1}^k\dfrac{1}{l!}D^l\theta_j(x)\left(\displaystyle\sum_{i=1}^k\varepsilon^i w_i^j(x)+O(\varepsilon^{k+1})\right)^l\right)\Bigg|_{\varepsilon=0}.\\
\end{array}
\]
According to the multilinearity of $D^l\theta_j(x)$ and the above expression, we have
\[
	\begin{array}{rcl}
		\alpha_j^q(x)&=&\dfrac{\partial^q}{\partial\varepsilon^q}\left(\displaystyle\sum_{l=1}^k\dfrac{1}{l!}D^l\theta_j(x)\left(\displaystyle\sum_{i=1}^k\varepsilon^i w_i^j(x))\right)^l+O(\varepsilon^{k+1})\right)\Bigg|_{\varepsilon=0}\\
		&=&\dfrac{\partial^q}{\partial\varepsilon^q}\left(\displaystyle\sum_{l=1}^k\dfrac{1}{l!}D^l\theta_j(x)\left(\displaystyle\sum_{i=1}^k\varepsilon^i w_i^j(x))\right)^l\right)\Bigg|_{\varepsilon=0}.
	\end{array}
\]
From the above expression and Lemma \ref{Ql}, it writes
\begin{equation}\label{alphaauxsd}
\alpha_j^q(x)=\displaystyle\sum_{l=1}^k\dfrac{1}{l!}\displaystyle\sum_{p=l}^{kl}\dfrac{\partial^q}{\partial\varepsilon^q}\left(\varepsilon^p\right)\Big|_{\varepsilon=0}\displaystyle\sum_{u\in S_{p,l}}D^l\theta_j(x)\left(\displaystyle\prod_{r=1}^l w_{u_r}^j(x)\right),
\end{equation}
where $S_{p,l}=\left\lbrace(u_1,\dots,u_l)\in (\mathbb{N}^*)^l:u_1+\dots+u_l=p\right\rbrace$. 

Notice that
$$
\dfrac{\partial^q}{\partial\varepsilon^q}\left(\varepsilon^b\right)\Big|_{\varepsilon=0}=\left\lbrace\begin{array}{rl}
	q!,& p=q,\\
	0,& p\neq q.
\end{array}\right.$$
Thus, from this and \eqref{alphaauxsd}, we obtain
\begin{equation}\label{dsalpha}
	\begin{array}{rl}
		\alpha_j^q(x)=\displaystyle\sum_{l=q}^k\dfrac{q!}{l!}\displaystyle\sum_{u\in S_{q,l}}D^l\theta_j(x)\left(\displaystyle\prod_{r=1}^l w_{u_r}^j(x)\right).
	\end{array}
\end{equation}

Note that if $q<l$ and exist $(b_1,\dots,b_l)\in S_{q,l}$, then
$l\leq\displaystyle\sum_{t=1}^lb_t=q<l.$
It is a contradiction, Thus, $S_{q,l}$ is empty for $q<l$. Then, from this fact, \eqref{wchapij}, and \eqref{dsalpha}, it is writes
\[
\alpha_j^q(x)=\displaystyle\sum_{l=1}^q\dfrac{q!}{l!}\displaystyle\sum_{u\in S_{q,l}}D^l\theta_j(x)\left(\displaystyle\prod_{r=1}^l w_{u_r}^j(x)\right),
\]
where
\[		w_1^j(x)={z}_{1}^{j-1}(\theta_j(x),x)\]
and applying the formula's Faà di Bruno in \eqref{wchapij}, we have
\[
	\begin{array}{rcl}
		w_i^j(x)&=&\dfrac{1}{i!}z_{i}^{j-1}(\theta_j(x),x)\\
		& &+\displaystyle\sum_{a=1}^{i-1}\displaystyle\sum_{b\in S_a}\dfrac{1}{(i-a)!b_1!b_2!2!^{b_2}\dots b_a!a!^{b_a}}\partial_t^{L_b}z_{i-a}^{j-1}(\theta_j(x),x)\displaystyle\prod_{m=1}^a\left(\alpha_j^m(x) \right)^{b_m}.
	\end{array}
\]
This finishes the proof of Claim \ref{claim1}.
\end{proof}

\section*{Acknowledgements}

We thank the anonymous referee for his/her comments that led to an improved version of this paper, in special for bringing the interesting references \cite{Han17,HS15,HY21,Lh10} to our attention.

We also thank Espa\c{c}o da Escrita -- Pr\'{o}-Reitoria de Pesquisa -- UNICAMP for the language services provided.

KSA is partially supported by CAPES grant 88887.308850/2018-00.
DRC is partially supported by CAPES grant 001.
DDN is partially supported by FAPESP grants 2018/16430-8, 2018/ 13481-0, and 2019/10269-3, and by CNPq grant 306649/2018-7. 
KSA, OARC, and DDN are partially supported by CNPq grant 438975/2018-9. 	  

\bibliographystyle{abbrv}
\bibliography{references.bib}

\end{document}